\documentclass[12pt,reqno]{amsart}
\usepackage[margin=1in]{geometry}
\usepackage{amsmath,amssymb,amsthm,graphicx,amsxtra, setspace}
\usepackage[utf8]{inputenc}
\usepackage{mathrsfs}
\usepackage{hyperref}
\usepackage{upgreek}
\usepackage{mathtools}
\allowdisplaybreaks

\usepackage[pagewise]{lineno}

\newtheorem{theorem}{Theorem}[section]
\newtheorem{lemma}[theorem]{Lemma}
\newtheorem{proposition}[theorem]{Proposition}
\newtheorem{assumption}[theorem]{Assumption}
\newtheorem{corollary}[theorem]{Corollary}
\newtheorem{definition}[theorem]{Definition}

\newtheorem{remark}[theorem]{Remark}

\let\originalleft\left
\let\originalright\right
\renewcommand{\left}{\mathopen{}\mathclose\bgroup\originalleft}
\renewcommand{\right}{\aftergroup\egroup\originalright}

\newcommand{\Tr}{\mathop{\mathrm{Tr}}}
\renewcommand{\d}{\/\mathrm{d}\/}

\def\w{\textbf{W}^{\varepsilon}_{{\theta}^{\varepsilon}}}

\def\e{\varepsilon}

\def\t{t\wedge\tau_N}

\def\S{\mathrm{S}}
\def\T{T\wedge\tau_N}
\def\L{\mathbb{L}}
\def\A{\mathrm{A}}

\def\C{\mathrm{C}}
\def\f{\mathbf{f}}

\def\B{\mathrm{B}}
\def\D{\mathrm{D}}
\def\y{\mathbf{y}}

\def\E{\mathbb{E}}
\def\X{\mathbb{X}}
\def\x{\mathbf{x}}

\def\z{\mathbf{z}}
\def\v{\mathbf{v}}
\def\V{\mathbb{v}}
\def\w{\mathbf{w}}
\def\W{\mathrm{W}}
\def\G{\mathrm{G}}

\def\M{\mathrm{M}}
\def\N{\mathbb{N}}

\def\V{\mathbb{V}}
\def\wi{\widetilde}

\def\u{\mathrm{U}}
\def\P{\mathrm{P}}
\def\u{\mathbf{u}}
\def\H{\mathbb{H}}

\newcommand{\R}{\mathbb{R}}

\renewcommand{\d}{\/\mathrm{d}\/}

\newcommand{\Addresses}{{% additional braces for segregating \footnotesize
		\footnote{
			\noindent \textsuperscript{1}Department of Mathematics, Indian Institute of Technology Roorkee-IIT Roorkee,
			Haridwar Highway, Roorkee, Uttarakhand 247667, INDIA.\par\nopagebreak
			\noindent  \textit{e-mail:} \texttt{manilfma@iitr.ac.in, maniltmohan@gmail.com.}
			
			\noindent \textsuperscript{*}Corresponding author.

			\textit{Key words:} convective Brinkman-Forchheimer equations, degenerate noise, asymptotic log-Harnack inequality, asymptotic strong Feller property, invariant measure.
			
			Mathematics Subject Classification (2010): 60H15, 37A25, 35Q30, 35Q35.

}}}

\begin{document}
	
%	\linenumbers
	
	\title[Asymptotic log-Harnack inequality  for SCBF equations]{Asymptotic log-Harnack inequality for the stochastic convective Brinkman-Forchheimer equations with degenerate noise
			\Addresses}
	\author[M. T. Mohan ]{Manil T. Mohan\textsuperscript{1*}}

	\maketitle
	
	\begin{abstract}
	 In this work, we consider the two and three dimensional stochastic convective Brinkman-Forchheimer (SCBF)  equations and examine some asymptotic behaviors of its strong solution. We establish the asymptotic log-Harnack inequality for the transition semigroup associated with the SCBF equations driven by additive as well as multiplicative degenerate noise via the asymptotic coupling method.  As applications of the asymptotic log-Harnack inequality, we derive the gradient estimate, asymptotic irreducibility, asymptotic strong Feller property, asymptotic  heat kernel estimate and ergodicity. Whenever the absorption exponent  $r\in(3,\infty)$, the asymptotic log-Harnack inequality is obtained without any restriction on the Brinkman coefficient (effective viscosity) $\mu>0$, the Darcy coefficient $\alpha>0$ and the Forchheimer coefficient $\beta>0$.
	\end{abstract}

	\section{Introduction}\label{sec1}\setcounter{equation}{0}
The  convective Brinkman-Forchheimer (CBF) equations  describe the motion of  incompressible viscous fluid  through a rigid, homogeneous, isotropic, porous medium.	Let $\mathcal{O}\subset\R^n$ ($n=2,3$) be a bounded domain with a smooth boundary $\partial\mathcal{O}$. Let  $\u(t , x) \in \R^n$ denotes the velocity field at time $t\in[0,T]$ and position $x\in\mathcal{O}$, $p(t,x)\in\R$ represents the pressure field, $\f(t,x)\in\R^n$ stands for an external forcing.  Let the constant $\mu$ represents the positive Brinkman coefficient (effective viscosity), the positive constants $\alpha$ and $\beta$ denote the Darcy (permeability of porous medium) and Forchheimer (proportional to the porosity of the material) coefficients, respectively.  The	CBF  equations  are given by (see \cite{KT2} for Brinkman-Forchheimer equations with fast growing nonlinearities)
\begin{equation}\label{1}
\left\{
	\begin{aligned}
	\frac{\partial \u}{\partial t}-\mu \Delta\u+(\u\cdot\nabla)\u+\alpha\u+\beta|\u|^{r-1}\u+\nabla p&=\mathbf{f}, \ \text{ in } \ \mathcal{O}\times(0,T), \\ \nabla\cdot\u&=0, \ \text{ in } \ \mathcal{O}\times(0,T), \\
	\u&=\mathbf{0}\ \text{ on } \ \partial\mathcal{O}\times(0,T), \\
	\u(0)&=\x \ \text{ in } \ \mathcal{O}.
	\end{aligned}
	\right.
	\end{equation}
As mentioned earlier, the CBF equations \eqref{1} describe the motion of incompressible fluid flows in a saturated porous medium.  In order to obtain the uniqueness of the pressure $p$, one can impose the condition $
\int_{\mathcal{O}}p(x,t)\d x=0, $ for $t\in (0,T)$ also. The absorption exponent $r\in[1,\infty)$ and the case $r=3$ is known as the critical exponent.  Note that for $\alpha=\beta=0$, we obtain the classical 3D Navier-Stokes equations (see \cite{GGP,OAL,JCR3,Te,Te1}, etc). 

The Navier-Stokes problem in bounded domains with compact boundary, modified by the absorption term $|\u|^{r-2}\u$, for $r>2$ is considered in \cite{SNA}. The authors proved the existence of weak solutions in the Leray-Hopf sense, for any dimension $n\geq 2$ and its uniqueness for $n=2$. But  in three dimensions, the authors were not able to establish the energy equality satisfied by the weak solutions. The existence of regular dissipative solutions and global attractors for the system \eqref{1} in three dimensions with $r> 3$ is established in \cite{KT2}. For $r>3$, as a global smooth solution exists, the energy equality is satisfied by the weak solutions. Recently, the authors in \cite{CLF} were able to construct functions that can approximate functions defined on smooth bounded domains by elements of eigenspaces of linear operators (e.g., the Laplacian or the Stokes operator) in such a way that the approximations are bounded and converge in both Sobolev and Lebesgue spaces simultaneously. As a simple application of this result, they proved that all weak solutions of the critical CBF equations  ($r=3$) in three dimensions posed on a bounded domain in $\mathbb{R}^3$ satisfy the energy equality (see \cite{KWH} for the case of periodic domains).  The author in \cite{MTM7} proved the existence and uniqueness of a global weak solution in the Leray-Hopf sense satisfying the energy equality to the system \eqref{1}	 (in three dimensions, for all values of $\beta$ and $\mu $, whenever the absorption exponent $r>3$ and  $2\beta\mu \geq 1$, for the critical case $r=3$). The monotonicity as well as the demicontinuity properties of the linear and nonlinear operators and the Minty-Browder technique were exploited in the proofs. 

Let us now discuss some results available in the literature for the stochastic counterpart for the system \eqref{1}. The authors in \cite{MRTZ} showed the existence and uniqueness of strong solutions to the stochastic 3D tamed Navier-Stokes equations on bounded domains with Dirichlet boundary conditions. They also  proved a small time large deviation principle for the solution. By using classical Faedo-Galerkin approximation and compactness method, the existence of martingale solutions for the stochastic 3D Navier-Stokes equations with nonlinear damping is obtained in \cite{LHGH1}.  Recently, the author in \cite{MTM8} established the existence of a pathwise unique strong solution  satisfying the energy equality (It\^o's formula) to the stochastic convective Brinkman-Forchheimer (SCBF) equations perturbed by multiplicative Gaussian noise, by exploiting the monotonicity property of the linear and nonlinear operators as well as a stochastic generalization of  the Minty-Browder technique.  The author has also proved the existence of a unique ergodic and strongly mixing invariant measure for the SCBF  equations \eqref{1} subject to multiplicative Gaussian noise  (non-degenerate), by making use of the exponential stability of strong solutions. For a sample literature on stochastic tamed 3D Navier-Stokes equations and related models on periodic domains as well as on whole space, interested readers are referred to see \cite{HBAM,ZBGD,WLMR,WL,MRXZ1,MRTZ1}, etc and references therein. 

 To examine the diffusion semigroup on the Riemannian manifolds, the dimension-free Harnack inequality was introduced by Wang in \cite{FYW}. As this inequality was not available in many situations,  he introduced  log-Harnack inequality in \cite{FYW1}. Later, these two inequalities were widely used by many researchers in the context of stochastic partial differential equations (SPDEs)  to obtain the gradient estimates (and thus strong Feller property), irreducibility, heat kernel estimates, uniqueness of invariant probability measures, etc for the associated transition semigroups (cf. \cite{WLi,WLFY,MTM9,FYW3,FYW2,FYW4,FYW5}, etc and references therein).  It is observed that the strong Feller property may not hold for highly degenerate noise case.  The authors in \cite{MHJC} introduced the concept of asymptotically strong Feller property to examine the ergodic properties of the 2D Navier-Stokes equations perturbed by  degenerate noise.   The author in \cite{LXu} showed by an asymptotic coupling that the 2D stochastic Navier-Stokes equation driven by highly degenerate but essentially elliptic noise satisfies the asymptotic log-Harnack inequality, which implies the asymptotically strong Feller property.   The asymptotic log-Harnack inequality for  several  kinds  of  models  on  stochastic differential systems like non-degenerate SDEs, neutral SDEs, semi-linear SPDEs, and stochastic Hamiltonian systems with infinite memory is obtained in \cite{JBFY}.   By using the asymptotic coupling method, the asymptotic log-Harnack inequality for the transition semigroup associated to the 3D Leray-$\alpha$  model with fractional dissipation driven by highly degenerate noise is established in \cite{SLWL}. The asymptotic log-Harnack inequality and some of its consequent properties for a class of stochastic 2D hydrodynamical-type systems driven by degenerate noise are established in \cite{WHSL}. In \cite{WHSL1}, the authors established asymptotic log-Harnack inequality and discussed its applications for semilinear SPDEs with degenerate multiplicative noise by the coupling method. For a sample literature on the ergodic theory for the 2D  stochastic Navier-Stokes equations subjected to degenerate noise, we refer the interested readers to \cite{OBAK,ADe,WEJC,WEJC1,NGJC}, etc.  In this work, we consider the SCBF equations perturbed by additive as well as multiplicative degenerate noise and establish the asymptotic log-Harnack inequality  via the asymptotic coupling method. Then applying the results available in \cite{JBFY}, as applications of the asymptotic log-Harnack inequality, we also derive the gradient estimate, asymptotic irreducibility, asymptotic strong Feller property, asymptotic  heat kernel estimate and ergodicity results.  For $n=2,3$ and $r\in(3,\infty)$, we establish the asymptotic log-Harnack inequality  and ergodic properties of the transition semigroup associated with the SCBF equations without any restriction on the Brinkman coefficient (effective viscosity) $\mu$, the Darcy coefficient $\alpha$ and the Forchheimer coefficient $\beta$.

The rest of the paper is organized as follows. In the next section, we discuss about the global solvability of the stochastic counterpart of the system \eqref{1}.  We first provide the necessary function spaces needed to obtain the global solvability results of the stochastic system.   We also examine some porperties of the linear and nonlinear operators like monotone, demicontinuity, hemicontinuity, etc. The SCBF equations perturbed by degenerate Gaussian noise is also formulated in the same section. After providing an abstract formulation of the SCBF equations, we discuss about the existence and uniqueness of global strong solution. The asymptotic log-Harnack inequality for the SCBF equations perturbed additive, highly degenerate but essentially elliptic  noise is established in section \ref{sec4}. For $n=2$ and $r\in[1,3]$, we establish the asymptotic log-Harnack inequality for $\lambda_1\mu^3\geq 8\Tr(\sigma\sigma^*)$, where $\lambda_1$ is the first eigenvalue of the Stokes operator and $\sigma$ is the noise coefficient  (Theorem \ref{maint}). Regarding the applications of asymptotic log-Harnack inequality, we derive the gradient estimate, asymptotic irreducibility, asymptotic strong Feller property, asymptotic  heat kernel estimate and also the existence of a unique invariant measure (Corollary \ref{cor4.4}). For the case $n=2,3$ and $r\in(3,\infty)$, we obtain the asymptotic log-Harnack inequality for the transition semigroup associated with the SCBF equations for any $\mu,\beta>0$ and for the critical case ($n=r=3$), it has been established for $\beta\mu>1$ (Theorem \ref{maint1}).  In the final section, we consider the SCBF equations  subjected to multiplicative degenerate noise and establish the asymptotic log-Harnack inequality for $n=2,3$ and $r\in[3,\infty)$ ($\beta,\mu>0$ for $r\in(3,\infty)$ and $\beta\mu>1$ for $r=3$) (Theorem \ref{maint2}). 

\section{Stochastic convective  Brinkman-Forchheimer  equations}\label{sec2}\setcounter{equation}{0}
We discuss about the global solvability of the SCBF equations perturbed by additive degenerate noise in this section. We first provide the necessary function spaces needed to obtain the existence and uniqueness of strong solution.   In our analysis, the parameter $\alpha$ does not play a major role and we set $\alpha$ to be zero in \eqref{1} in the entire paper.

\subsection{Function spaces} Let $\C_0^{\infty}(\mathcal{O};\R^n)$ be the space of all infinitely differentiable functions  ($\R^n$-valued) with compact support in $\mathcal{O}\subset\R^n$.  Let us define 
\begin{align*} 
\mathcal{V}&:=\{\u\in\C_0^{\infty}(\mathcal{O},\R^n):\nabla\cdot\u=0\},\\
\mathbb{H}&:=\text{the closure of }\ \mathcal{V} \ \text{ in the Lebesgue space } \L^2(\mathcal{O})=\mathrm{L}^2(\mathcal{O};\R^n),\\
\mathbb{V}&:=\text{the closure of }\ \mathcal{V} \ \text{ in the Sobolev space } \H_0^1(\mathcal{O})=\mathrm{H}_0^1(\mathcal{O};\R^n),\\
\widetilde{\L}^{p}&:=\text{the closure of }\ \mathcal{V} \ \text{ in the Lebesgue space } \L^p(\mathcal{O})=\mathrm{L}^p(\mathcal{O};\R^n),
\end{align*}
for $p\in(2,\infty)$. Then under some smoothness assumptions on the boundary, we characterize the spaces $\H$, $\V$ and $\widetilde{\L}^p$ as 
$
\H=\{\u\in\L^2(\mathcal{O}):\nabla\cdot\u=0,\u\cdot\mathbf{n}\big|_{\partial\mathcal{O}}=0\}$,  with norm  $\|\u\|_{\H}^2:=\int_{\mathcal{O}}|\u(x)|^2\d x,
$
where $\mathbf{n}$ is the outward normal to $\partial\mathcal{O}$,
$
\V=\{\u\in\H_0^1(\mathcal{O}):\nabla\cdot\u=0\},$  with norm $ \|\u\|_{\V}^2:=\int_{\mathcal{O}}|\nabla\u(x)|^2\d x,
$ and $\widetilde{\L}^p=\{\u\in\L^p(\mathcal{O}):\nabla\cdot\u=0, \u\cdot\mathbf{n}\big|_{\partial\mathcal{O}}\},$ with norm $\|\u\|_{\widetilde{\L}^p}^p=\int_{\mathcal{O}}|\u(x)|^p\d x$, respectively.
Let $(\cdot,\cdot)$ represents the inner product in the Hilbert space $\H$ and $\langle \cdot,\cdot\rangle $ denotes the induced duality between the spaces $\V$  and its dual $\V'$ as well as $\widetilde{\L}^p$ and its dual $\widetilde{\L}^{p'}$, where $\frac{1}{p}+\frac{1}{p'}=1$. Note that $\H$ can be identified with its dual $\H'$. We endow the space $\V\cap\widetilde{\L}^{p}$ with the norm $\|\u\|_{\V}+\|\u\|_{\widetilde{\L}^{p}},$ for $\u\in\V\cap\widetilde{\L}^p$ and its dual $\V'+\widetilde{\L}^{p'}$ with the norm $$\inf\left\{\max\left(\|\v_1\|_{\V'},\|\v_1\|_{\widetilde{\L}^{p'}}\right):\v=\v_1+\v_2, \ \v_1\in\V', \ \v_2\in\widetilde{\L}^{p'}\right\}.$$ Furthermore, we have the continuous embedding $\V\cap\widetilde{\L}^p\hookrightarrow\H\hookrightarrow\V'+\widetilde{\L}^{p'}$. 

\subsection{Linear operator}
Let $\mathrm{P}_{\H} : \L^2(\mathcal{O}) \to\H$ denotes the \emph{Helmholtz-Hodge orthogonal projection}. We define
\begin{equation*}
\left\{
\begin{aligned}
\A\u:&=-\mathrm{P}_{\H}\Delta\u,\;\u\in\D(\A),\\ \D(\A):&=\V\cap\H^{2}(\mathcal{O}).
\end{aligned}
\right.
\end{equation*}
It can be easily seen that the operator $\A$ is a non-negative self-adjoint operator in $\H$ with $\V=\D(\A^{1/2})$ and \begin{align}\label{2.7a}\langle \A\u,\u\rangle =\|\u\|_{\V}^2,\ \textrm{ for all }\ \u\in\V, \ \text{ so that }\ \|\A\u\|_{\V'}\leq \|\u\|_{\V}.\end{align}
For a bounded domain $\mathcal{O}$, the operator $\A$ is invertible and its inverse $\A^{-1}$ is bounded, self-adjoint and compact in $\H$. Thus, using spectral theorem, the spectrum of $\A$ consists of an infinite sequence $0< \lambda_1\leq \lambda_2\leq\ldots\leq \lambda_k\leq \ldots,$ with $\lambda_k\to\infty$ as $k\to\infty$ of eigenvalues. 
\iffalse 
The behavior of these eigenvalues is well known in the literature (for example see Theorem 2.2, Corollary 2.2, \cite{AAI} and for asymptotic behavior, see \cite{PCCF,Te1,AAI1}, etc).  For all $k\geq 1$, we have  
\begin{align}\label{643}
\lambda_k\geq \widetilde{C}k^{2/n}, \ \text{ where }\ 
\widetilde{C}=\frac{n}{2+n}\left(\frac{(2\pi)^n}{\omega_n(n-1)|\mathcal{O}|}\right)^{2/n}, \ \omega_n= \pi^{n/2}\Gamma(1+n/2),
\end{align}  and $|\mathcal{O}|$ is the $n$-dimensional Lebesgue measure of $\mathcal{O}$. For $n=2$, we get $\widetilde{C}=\frac{4\sqrt{\pi}}{|\mathcal{O}|}$ and for $n=3$, we find  $\widetilde{C}=\frac{3^{5/3}\pi^{4/3}k^{2/3}}{5|\mathcal{O}|^{2/3}}$.  
\fi 
Moreover, there exists an orthogonal basis $\{e_k\}_{k=1}^{\infty} $ of $\H$ consisting of eigenvectors of $\A$ such that $\A e_k =\lambda_ke_k$,  for all $ k\in\mathbb{N}$.  We know that any $\u\in\H$ can be expressed as $\u=\sum_{k=1}^{\infty}\langle\u,e_k\rangle e_k$ and $\A\u=\sum_{k=1}^{\infty}\lambda_k\langle\u,e_k\rangle e_k$, for $\u\in\D(\A)$. Thus, it is immediate that 
\begin{align}\label{poin}
\|\nabla\u\|_{\mathbb{H}}^2=\langle \A\u,\u\rangle =\sum_{k=1}^{\infty}\lambda_k|\langle \u,e_k\rangle|^2\geq \lambda_1\sum_{k=1}^{\infty}|\langle\u,e_k\rangle|^2=\lambda_1\|\u\|_{\mathbb{H}}^2, \ \text{ for }\ \u\in\V. 
\end{align}

\subsection{Bilinear operator}
Let us define the \emph{trilinear form} $b(\cdot,\cdot,\cdot):\V\times\V\times\V\to\R$ by $$b(\u,\v,\w)=\int_{\mathcal{O}}(\u(x)\cdot\nabla)\v(x)\cdot\w(x)\d x=\sum_{i,j=1}^n\int_{\mathcal{O}}\u_i(x)\frac{\partial \v_j(x)}{\partial x_i}\w_j(x)\d x.$$ If $\u, \v$ are such that the linear map $b(\u, \v, \cdot) $ is continuous on $\V$, the corresponding element of $\V'$ is denoted by $\B(\u, \v)$. We also denote (with an abuse of notation) $\B(\u) = \B(\u, \u)=\mathrm{P}_{\H}(\u\cdot\nabla)\u$.
An integration by parts gives 
\begin{equation}\label{b0}
\left\{
\begin{aligned}
b(\u,\v,\v) &= 0,\text{ for all }\u,\v \in\V,\\
b(\u,\v,\w) &=  -b(\u,\w,\v),\text{ for all }\u,\v,\w\in \V.
\end{aligned}
\right.\end{equation}
In the trilinear form, an application of H\"older's inequality yields
\begin{align*}
|b(\u,\v,\w)|=|b(\u,\w,\v)|\leq \|\u\|_{\widetilde{\L}^{r+1}}\|\v\|_{\widetilde{\L}^{\frac{2(r+1)}{r-1}}}\|\w\|_{\V},
\end{align*}
for all $\u\in\V\cap\widetilde{\L}^{r+1}$, $\v\in\V\cap\widetilde{\L}^{\frac{2(r+1)}{r-1}}$ and $\w\in\V$, so that we get 
\begin{align}\label{2p9}
\|\B(\u,\v)\|_{\V'}\leq \|\u\|_{\widetilde{\L}^{r+1}}\|\v\|_{\widetilde{\L}^{\frac{2(r+1)}{r-1}}}.
\end{align}
Hence, the trilinear map $b : \V\times\V\times\V\to \R$ has a unique extension to a bounded trilinear map from $(\V\cap\widetilde{\L}^{r+1})\times(\V\cap\widetilde{\L}^{\frac{2(r+1)}{r-1}})\times\V$ to $\R$. It can also be seen that $\B$ maps $ \V\cap\widetilde{\L}^{r+1}$  into $\V'+\widetilde{\L}^{\frac{r+1}{r}}$ and using interpolation inequality, we get 
\begin{align}\label{212}
\left|\langle \B(\u,\u),\v\rangle \right|=\left|b(\u,\v,\u)\right|\leq \|\u\|_{\widetilde{\L}^{r+1}}\|\u\|_{\widetilde{\L}^{\frac{2(r+1)}{r-1}}}\|\v\|_{\V}\leq\|\u\|_{\widetilde{\L}^{r+1}}^{\frac{r+1}{r-1}}\|\u\|_{\H}^{\frac{r-3}{r-1}}\|\v\|_{\V},
\end{align}
for all $\v\in\V\cap\widetilde{\L}^{r+1}$. Thus, we have 
\begin{align}\label{2.9a}
\|\B(\u)\|_{\V'+\widetilde{\L}^{\frac{r+1}{r}}}\leq\|\u\|_{\widetilde{\L}^{r+1}}^{\frac{r+1}{r-1}}\|\u\|_{\H}^{\frac{r-3}{r-1}},
\end{align}
for $r\geq 3$. 
\iffalse 
Note that $\B$ also maps $\widetilde{\L}^6\cap \H$ into $\V'$ and
\begin{align*}
\left|\langle \B(\u,\u),\v\rangle \right|=\left|b(\u,\v,\u)\right|\leq \|\u\|_{\widetilde{\L}^3}\|\u\|_{\widetilde{\L}^6}\|\nabla\v\|_{\H}
\end{align*}
so that 
\begin{align}
\|\B(\u,\u)\|_{\V'}\leq\|\u\|_{\widetilde{\L}^3}\|\u\|_{\widetilde{\L}^6}\leq C\|\u\|_{\H}^{1/2}\|\u\|_{\V}^{3/2},
\end{align}
once again, we get an estimate similar to \eqref{2.9a}. 
Using \eqref{2p9}, for $\u,\v\in\V\cap\widetilde{\L}^{r+1}$, we also have 
\begin{align}\label{lip}
\|\B(\u)-\B(\v)\|_{\V'+\widetilde{\L}^{\frac{r+1}{r}}}&\leq \|\B(\u-\v,\u)\|_{\V'}+\|\B(\v,\u-\v)\|_{\V'}\nonumber\\&\leq \left(\|\u\|_{\widetilde{\L}^{\frac{2(r+1)}{r-1}}}+\|\v\|_{\widetilde{\L}^{\frac{2(r+1)}{r-1}}}\right)\|\u-\v\|_{\widetilde{\L}^{r+1}}\nonumber\\&\leq \left(\|\u\|_{\H}^{\frac{r-3}{r-1}}\|\u\|_{\widetilde{\L}^{r+1}}^{\frac{2}{r-1}}+\|\v\|_{\H}^{\frac{r-3}{r-1}}\|\v\|_{\widetilde{\L}^{r+1}}^{\frac{2}{r-1}}\right)\|\u-\v\|_{\widetilde{\L}^{r+1}},
\end{align}
for $r>3$, by using the interpolation inequality. For $r=3$, a calculation similar to \eqref{lip} yields 
\begin{align}
\|\B(\u)-\B(\v)\|_{\V'+\widetilde{\L}^{\frac{4}{3}}}&\leq \left(\|\u\|_{\widetilde{\L}^{4}}+\|\v\|_{\widetilde{\L}^{4}}\right)\|\u-\v\|_{\widetilde{\L}^{4}},
\end{align}
  hence $\B(\cdot):\V\cap\widetilde{\L}^{4}\to\V'+\widetilde{\L}^{\frac{4}{3}}$ is a locally Lipschitz operator. 
  \fi 
  
  For $n=2$ and $r\in[1,3]$, using H\"older's and Ladyzhenskaya's inequalities, we obtain 
  \begin{align*}
  |\langle\B(\u,\v),\w\rangle|=|\langle\B(\u,\w),\v\rangle|\leq\|\u\|_{\wi\L^4}\|\v\|_{\wi\L^4}\|\w\|_{\V},
  \end{align*}
  for all $\u,\v\in\wi\L^4$ and $\w\in\V$, so that we get $\|\B(\u,\v)\|_{\V'}\leq\|\u\|_{\wi\L^4}\|\v\|_{\wi\L^4}$. Furthermore, we have $$\|\B(\u,\u)\|_{\V'}\leq\|\u\|_{\wi\L^4}^2\leq\sqrt{2}\|\u\|_{\H}\|\u\|_{\V}\leq\sqrt{\frac{2}{\lambda_1}}\|\u\|_{\V}^2,$$ for all $\u\in\V$. 
\iffalse 
For $r\geq 2$, using an integration by parts, boundary condition and divergence free condition, we estimate $\langle \B(\u),|\u|^{r-2}\u\rangle$ as 
\begin{align}\label{2.12}
\langle \B(\u),|\u|^{r-2}\u\rangle&=\langle (\u\cdot\nabla)\u,|\u|^{r-2}\u\rangle=\sum_{i,j,k=1}^n\int_{\mathcal{O}}\u_i(x)\frac{\partial \u_j(x)}{\partial x_i}|\u_k(x)|^{r-2}\u_j(x)\d x\nonumber\\&=\frac{1}{2}\int_{\mathcal{O}}\u_i(x)|\u_k(x)|^{r-2}\frac{\partial\u_j^2(x)}{\partial x_i}\d x\nonumber\\&=\frac{1}{2}\sum_{i,j,k=1}^n\bigg\{\int_{\partial\mathcal{O}}\u_i(x)|\u_k(x)|^{r-2}\u_j^2(x)\mathbf{n}_i(x)\d x-\int_{\mathcal{O}}\frac{\partial\u_i(x)}{\partial x_i}|\u_k(x)|^{r-2}\u_j^2(x)\d x \nonumber\\&\qquad-\int_{\mathcal{O}}\frac{\partial|\u_k(x)|^{r-2}}{\partial x_i}\u_i(x)\u_j^2(x)\d x\bigg\}\nonumber\\&=-\frac{1}{2}\sum_{i=1}^n\int_{\mathcal{O}}\frac{\partial|\u(x)|^{r-2}}{\partial x_i}|\u(x)|^2\u_i(x)\d x\nonumber\\&=-\frac{(r-2)}{2}\sum_{i,j=1}^n\int_{\mathcal{O}}|\u(x)|^{r-2}\u_i(x)\frac{\partial\u_j(x)}{\partial x_i}\u_j(x)\d x\nonumber\\&=-\frac{(r-2)}{2}\langle \B(\u),|\u|^{r-2}\u\rangle,
\end{align}
so that we get $\langle \B(\u),|\u|^{r-2}\u\rangle=0$. 
\fi 
\subsection{Nonlinear operator}
Let us now consider the operator $\mathcal{C}(\u):=\P_{\H}(|\u|^{r-1}\u)$. It is immediate that $\langle\mathcal{C}(\u),\u\rangle =\|\u\|_{\widetilde{\L}^{r+1}}^{r+1}$. 
\iffalse 
 and the map $\mathcal{C}(\cdot):\widetilde{\L}^{r+1}\to\widetilde{\L}^{\frac{r+1}{r}}$ is Fr\'echet differentiable with Fr\'echet derivative $\mathcal{C}'(\u)\v=r\mathrm{P}_{\H}(|\u|^{r-1}\v),$ for $\v\in\widetilde{\L}^{r+1}$. For  $\u,\v\in\widetilde{\L}^{r+1}$, using Taylor's formula, we have 
\begin{align}\label{213}
\langle \P_{\H}(|\u|^{r-1}\u)-\P_{\H}(|\v|^{r-1}\v),\w\rangle&\leq \|(|\u|^{r-1}\u)-(|\v|^{r-1}\v)\|_{\widetilde{\L}^{\frac{r+1}{r}}}\|\w\|_{\widetilde{\L}^{r+1}}\nonumber\\&\leq \sup_{0<\theta<1}r\|(\u-\v)|\theta\u+(1-\theta)\v|^{r-1}\|_{\widetilde{\L}^{{\frac{r+1}{r}}}}\|\w\|_{\widetilde{\L}^{r+1}}\nonumber\\&\leq \sup_{0<\theta<1} r\|\theta\u+(1-\theta)\v\|_{\widetilde{\L}^{r+1}}^{r-1}\|\u-\v\|_{\widetilde{\L}^{r+1}}\|\w\|_{\widetilde{\L}^{r+1}}\nonumber\\&\leq r\left(\|\u\|_{\widetilde{\L}^{r+1}}+\|\v\|_{\widetilde{\L}^{r+1}}\right)^{r-1}\|\u-\v\|_{\widetilde{\L}^{r+1}}\|\w\|_{\widetilde{\L}^{r+1}},
\end{align}
for all $\u,\v,\w\in\widetilde{\L}^{r+1}$. 
Thus the operator $\mathcal{C}(\cdot):\widetilde{\L}^{r+1}\to\widetilde{\L}^{\frac{r+1}{r}}$ is locally Lipschitz.
\fi  
\iffalse 
\begin{align}\label{2p13}
&\langle \mathrm{P}_{\H}(\u|\u|^{r-1})-\mathrm{P}_{\H}(\v|\v|^{r-1}),\u-\v\rangle\nonumber\\&= r\langle (\u-\v)|\theta\u+(1-\theta)\v|^{r-1},\u-\v\rangle \nonumber\\&=r\int_{\mathcal{O}} |\theta\u(x)+(1-\theta)\v(x)|^{r-1}|\u(x)-\v(x)|^2\d x\geq 0,
\end{align}
for all $\u,\v\in\V\cap\widetilde{\L}^{r+1}$. The above estimate can be proved in the following way also:
\fi 
For any $r\in[1,\infty)$, we have 
\begin{align}\label{2pp11}
&\langle \mathrm{P}_{\H}(\u|\u|^{r-1})-\mathrm{P}_{\H}(\v|\v|^{r-1}),\u-\v\rangle\nonumber\\&=\int_{\mathcal{O}}\left(\u(x)|\u(x)|^{r-1}-\v(x)|\v(x)|^{r-1}\right)\cdot(\u(x)-\v(x))\d x\nonumber\\&=\int_{\mathcal{O}}\left(|\u(x)|^{r+1}-|\u(x)|^{r-1}\u(x)\cdot\v(x)-|\v(x)|^{r-1}\u(x)\cdot\v(x)+|\v(x)|^{r+1}\right)\d x\nonumber\\&\geq\int_{\mathcal{O}}\left(|\u(x)|^{r+1}-|\u(x)|^{r}|\v(x)|-|\v(x)|^{r}|\u(x)|+|\v(x)|^{r+1}\right)\d x\nonumber\\&=\int_{\mathcal{O}}\left(|\u(x)|^r-|\v(x)|^r\right)(|\u(x)|-|\v(x)|)\d x\geq 0. 
\end{align}
Furthermore, we find 
\begin{align}\label{224}
&\langle\mathrm{P}_{\H}(\u|\u|^{r-1})-\mathrm{P}_{\H}(\v|\v|^{r-1}),\u-\v\rangle\nonumber\\&=\langle|\u|^{r-1},|\u-\v|^2\rangle+\langle|\v|^{r-1},|\u-\v|^2\rangle+\langle\v|\u|^{r-1}-\u|\v|^{r-1},\u-\v\rangle\nonumber\\&=\||\u|^{\frac{r-1}{2}}(\u-\v)\|_{\H}^2+\||\v|^{\frac{r-1}{2}}(\u-\v)\|_{\H}^2\nonumber\\&\quad+\langle\u\cdot\v,|\u|^{r-1}+|\v|^{r-1}\rangle-\langle|\u|^2,|\v|^{r-1}\rangle-\langle|\v|^2,|\u|^{r-1}\rangle.
\end{align}
But, we know that 
\begin{align*}
&\langle\u\cdot\v,|\u|^{r-1}+|\v|^{r-1}\rangle-\langle|\u|^2,|\v|^{r-1}\rangle-\langle|\v|^2,|\u|^{r-1}\rangle\nonumber\\&=-\frac{1}{2}\||\u|^{\frac{r-1}{2}}(\u-\v)\|_{\H}^2-\frac{1}{2}\||\v|^{\frac{r-1}{2}}(\u-\v)\|_{\H}^2+\frac{1}{2}\langle\left(|\u|^{r-1}-|\v|^{r-1}\right),\left(|\u|^2-|\v|^2\right)\rangle \nonumber\\&\geq -\frac{1}{2}\||\u|^{\frac{r-1}{2}}(\u-\v)\|_{\H}^2-\frac{1}{2}\||\v|^{\frac{r-1}{2}}(\u-\v)\|_{\H}^2.
\end{align*}
From \eqref{224}, we finally have 
\begin{align}\label{2.23}
&\langle\mathrm{P}_{\H}(\u|\u|^{r-1})-\mathrm{P}_{\H}(\v|\v|^{r-1}),\u-\v\rangle\geq \frac{1}{2}\||\u|^{\frac{r-1}{2}}(\u-\v)\|_{\H}^2+\frac{1}{2}\||\v|^{\frac{r-1}{2}}(\u-\v)\|_{\H}^2\geq 0,
\end{align}
for $r\geq 1$.  	It is important to note that 
\begin{align}\label{a215}
\|\u-\v\|_{\wi\L^{r+1}}^{r+1}&=\int_{\mathcal{O}}|\u(x)-\v(x)|^{r-1}|\u(x)-\v(x)|^{2}\d x\nonumber\\&\leq 2^{r-2}\int_{\mathcal{O}}(|\u(x)|^{r-1}+|\v(x)|^{r-1})|\u(x)-\v(x)|^{2}\d x\nonumber\\&\leq 2^{r-2}\||\u|^{\frac{r-1}{2}}(\u-\v)\|_{\L^2}^2+2^{r-2}\||\v|^{\frac{r-1}{2}}(\u-\v)\|_{\L^2}^2. 
\end{align}
Combining \eqref{2.23} and \eqref{a215}, we obtain 
\begin{align}\label{214}
\langle\mathcal{C}(\u)-\mathcal{C}(\v),\u-\v\rangle\geq\frac{1}{2^{r-1}}\|\u-\v\|_{\wi\L^{r+1}}^{r+1},
\end{align}
for $r\geq 1$. 
\iffalse 
For the critical case $r=3$, we have 
\begin{align}\label{2p23}
&\langle \mathrm{P}_{\H}(\u|\u|^{2})-\mathrm{P}_{\H}(\v|\v|^{2}),\u-\v\rangle\nonumber\\&=\langle \u|\u|^2-\v|\v|^2,\u-\v\rangle =\langle\u(|\u|^2-|\v|^2)+(\u-\v)|\v|^2,\u-\v\rangle \nonumber\\&=\|\v(\u-\v)\|_{\H}^2+\langle\u\cdot(\u-\v),(\u+\v)\cdot(\u-\v)\rangle \nonumber\\&=\|\v(\u-\v)\|_{\H}^2+\|\u\cdot(\u-\v)\|_{\H}^2+\langle\u\cdot(\u-\v),\v\cdot(\u-\v)\rangle \nonumber\\&\geq \|\v(\u-\v)\|_{\H}^2+\|\u\cdot(\u-\v)\|_{\H}^2-\|\u\cdot(\u-\v)\|_{\H}\|\v(\u-\v)\|_{\H}\nonumber\\&\geq\frac{1}{2}\|\v(\u-\v)\|_{\H}^2+\frac{1}{2}\|\u\cdot(\u-\v)\|_{\H}^2\geq \frac{1}{2}\|\v(\u-\v)\|_{\H}^2\geq 0,
\end{align}
where we used Young's inequality.

Applying integration by parts, we get 
\begin{align}\label{219}
(\mathcal{C}(\u),\A\u)&=\sum_{i,j=1}^n\int_{\mathcal{O}}|\u(x)|^{r-1}\u_j(x)\frac{\partial^2\u_j(x)}{\partial x_i^2}\d x\nonumber\\&= -\sum_{i,j=1}^n\int_{\mathcal{O}}|\u(x)|^{r-1}\left(\frac{\partial\u_j(x)}{\partial x_i}\right)^2\d x-\sum_{i,j=1}^n\int_{\mathcal{O}}\frac{\partial}{\partial x_i}|\u(x)|^{r-1}\u_j(x)\frac{\partial\u_j(x)}{\partial x_i}\d x\nonumber\\&=-(r-1)\sum_{i,j=1}^n\int_{\mathcal{O}}|\u(x)|^{r-1}\left(\frac{\partial\u_j(x)}{\partial x_i}\right)^2\d x=-(r-1)\||\u|^{\frac{r-1}{2}}|\nabla\u|\|_{\H}^2.
\end{align}

\fi 
\subsection{Monotonicity}
In this subsection, we discuss about the monotonicity as well as the hemicontinuity properties of the linear and nonlinear operators.
\begin{definition}[\cite{VB}]
	Let $\X$ be a Banach space and let $\X^{'}$ be its topological dual.
	An operator $\G:\mathrm{D}\rightarrow
	\X^{'},$ $\mathrm{D}=\mathrm{D}(\G)\subset \X$ is said to be
	\emph{monotone} if
	$$\langle\G(x)-\G(y),x-y\rangle\geq
	0,\ \text{ for all } \ x,y\in \mathrm{D}.$$ \iffalse 
	The operator $\G(\cdot)$ is \emph{maximal monotone} if there is no monotone operator that properly contains it,  that is, if for $x\in\X$ and $w\in\X'$, the inequality $\langle w-\G(x),x-y\rangle\geq 0$, for all $y\in\X$ implies $w=\G(x)$.
	\fi 
	The operator $\G(\cdot)$ is said to be \emph{hemicontinuous}, if for all $x, y\in\X$ and $w\in\X',$ $$\lim_{\lambda\to 0}\langle\G(x+\lambda y),w\rangle=\langle\G(x),w\rangle.$$
	The operator $\G(\cdot)$ is called \emph{demicontinuous}, if for all $x\in\mathrm{D}$ and $y\in\X$, the functional $x \mapsto\langle \G(x), y\rangle$  is continuous, or in other words, $x_k\to x$ in $\X$ implies $\G(x_k)\xrightarrow{w}\G(x)$ in $\X'$. Clearly demicontinuity implies hemicontinuity. 
\end{definition}
\begin{theorem}[Theorem 2.2, \cite{MTM7}]\label{thm2.2}
	Let $\u,\v\in\V\cap\widetilde{\L}^{r+1}$, for $r>3$. Then,	for the operator $\G(\u)=\mu \A\u+\B(\u)+\beta\mathcal{C}(\u)$, we  have 
	\begin{align}\label{fe}
	\langle(\G(\u)-\G(\v),\u-\v\rangle+\eta\|\u-\v\|_{\H}^2\geq 0,
	\end{align}
	where \begin{align}\label{215}\eta=\frac{r-3}{2\mu(r-1)}\left(\frac{2}{\beta\mu (r-1)}\right)^{\frac{2}{r-3}}.\end{align} That is, the operator $\G+\eta\mathrm{I}$ is a monotone operator from $\V\cap\widetilde{\L}^{r+1}$ to $\V'+\widetilde{\L}^{\frac{r+1}{r}}$. 
\end{theorem}

\begin{theorem}[Theorem 2.3, \cite{MTM7}]\label{thm2.3}
	For the critical case $r=3$ with $2\beta\mu \geq 1$, the operator $\G(\cdot):\V\cap\widetilde{\L}^{r+1}\to \V'+\widetilde{\L}^{\frac{r+1}{r}}$ is globally monotone, that is, for all $\u,\v\in\V$, we have 
\begin{align}\label{218}\langle\G(\u)-\G(\v),\u-\v\rangle\geq 0.\end{align}
	\end{theorem}
\begin{theorem}[Remark 2.4, \cite{MTM7}]
	Let $n=2$, $r\in[1,3]$ and $\u,\v\in\V$. Then,	for the operator $\G(\u)=\mu \A\u+\B(\u)+\beta\mathcal{C}(\u)$, we  have 
	\begin{align}\label{fe2}
	\langle(\G(\u)-\G(\v),\u-\v\rangle+ \frac{27}{32\mu ^3}N^4\|\u-\v\|_{\H}^2\geq 0,
	\end{align}
	for all $\v\in{\mathbb{B}}_N$, where ${\mathbb{B}}_N$ is an $\widetilde{\L}^4$-ball of radius $N$, that is,
	$
	{\mathbb{B}}_N:=\big\{\z\in\widetilde{\L}^4:\|\z\|_{\widetilde{\L}^4}\leq N\big\}.
	$
	\end{theorem}

\begin{lemma}[Lemma 2.5, \cite{MTM7}]\label{lem2.8}
	The operator $\G:\V\cap\widetilde{\L}^{r+1}\to \V'+\widetilde{\L}^{\frac{r+1}{r}}$ is demicontinuous. 
\end{lemma}

\subsection{Stochastic convective  Brinkman-Forchheimer  equations}\label{sec3}
Let $(\Omega,\mathscr{F},\mathbb{P})$ be a complete probability space equipped with an increasing family of sub-sigma fields $\{\mathscr{F}_t\}_{0\leq t\leq T}$ of $\mathscr{F}$ satisfying:
\begin{enumerate}
	\item [(i)] $\mathscr{F}_0$ contains all elements $F\in\mathscr{F}$ with $\mathbb{P}(F)=0$,
	\item [(ii)] $\mathscr{F}_t=\mathscr{F}_{t+}=\bigcap\limits_{s>t}\mathscr{F}_s,$ for $0\leq t\leq T$.
\end{enumerate}  
We consider the following stochastic  convective Brinkman-Forchheimer  equations perturbed by additive degenerate Gaussian noise:
\begin{equation}\label{31}
\left\{
\begin{aligned}
\d\u(t)-\mu \Delta\u(t)&+(\u(t)\cdot\nabla)\u(t)+\beta|\u(t)|^{r-1}\u(t)+\nabla p(t)\\&=\sigma\d\W(t), \ \text{ in } \ \mathcal{O}\times(0,T), \\ \nabla\cdot\u(t)&=0, \ \text{ in } \ \mathcal{O}\times(0,T), \\
\u(t)&=\mathbf{0},\ \text{ on } \ \partial\mathcal{O}\times(0,T), \\
\u(0)&=\x, \ \text{ in } \ \mathcal{O},
\end{aligned}
\right.
\end{equation} 
where $\W(\cdot)$  is a cylindrical Wiener process defined on $(\Omega,\mathscr{F},\{\mathscr{F}_t\}_{t\geq 0},\mathbb{P})$ taking values in the Hilbert space $\H$. 
On  taking orthogonal projection $\mathrm{P}_{\H}$ onto the first equation in \eqref{31}, we get 
\begin{equation}\label{32}
\left\{
\begin{aligned}
\d\u(t)+[\mu \A\u(t)+\B(\u(t))+\beta\mathcal{C}(\u(t))]\d t&=\sigma\d\W(t), \ t\in(0,T),\\
\u(0)&=\x,
\end{aligned}
\right.
\end{equation}
where $\x\in\H$. For any given $N\in\N$, we define the projection $\mathrm{P}_N:\H\to\H_N$ by 
$$\mathrm{P}_N\u:=\sum_{k=1}^N(\u,e_k) e_k, \ \u\in\H, $$  where $\{e_k\}_{k=1}^{\infty}$ is an orthonormal basis in $\H$ and $\H_N=\mathrm{span}\{e_1,\ldots,e_N\}$. In this section, we assume that  the noise co-efficient $\sigma$ satisfies the following highly degenerate but essentially elliptic condition provided in  Section 4.5 of \cite{MHJC}. 
\begin{assumption}\label{ass4.1}
	There exists a sufficiently large but fixed $N_0\in\mathbb{N}$ such that $\mathrm{Range}(\sigma)=\mathrm{P}_{N_0}\H$ and $\sigma\u=\mathbf{0}$ if $\u\in(\mathrm{I}-\mathrm{P}_{N_0})$, where $\mathrm{I}$ is the identity mapping. 
\end{assumption}
For the fixed $N_0$ defined in the Assumption \ref{ass4.1}, we separate the Hilbert space $\H$ into the low and high frequency parts as 
\begin{align*}
\H=\mathrm{P}_{N_0}\H+(\mathrm{I}-\mathrm{P}_{N_0})\H, 
\end{align*}
and we denote $\H^l:=\mathrm{P}_{N_0}\H$ and $\H^h:=(\mathrm{I}-\mathrm{P}_{N_0})\H$. For any $\u\in\H$, we define $\u^l:=\mathrm{P}_{N_0}\u$  and $\u^h:=(\mathrm{I}-\mathrm{P}_{N_0})\u$. The following lemma is easy to prove and one can get a proof from Lemma 3.2, \cite{WHSL}. 
\begin{lemma}
	For any $\alpha>0$, we have 
	\begin{align}\label{4.3}
	\|\u^l\|_{\V}\leq \lambda_{N_0}^{1/2}\|\u^l\|_{\H}, \ \|\u^h\|_{\V}\geq \lambda_{N_0}^{1/2}\|\u^h\|_{\H}. 
	\end{align}
\end{lemma} From the Assumption \ref{ass4.1}, it is clear that $\Tr(\sigma\sigma^*)<\infty$ and $\sigma:\H^l\to\H^l$ is invertible, that is, there exists a constant $C_{\sigma}>0$ such that 
\begin{align}\label{4.1}
\|\sigma^{-1}\u\|_{\H}\leq C_{\sigma}\|\u\|_{\H}, \ \text{ for all }\ \u\in\H^l. 
\end{align}
In this work, we choose a large (but fixed) $N_0$ to make the noise has essential ellipticity effect (see Section 4.5, \cite{MHJC}).

\subsection{Global strong solution}
In this subsection, we give the definition of a unique global strong solution in the probabilistic sense to the system (\ref{32}). 
\begin{definition}[Global strong solution]
	Let $\x\in\H$ be given. An $\H$-valued $(\mathscr{F}_t)_{t\geq 0}$-adapted stochastic process $\u(\cdot)$ is called a \emph{strong solution} to the system (\ref{32}) if the following conditions are satisfied: 
	\begin{enumerate}
		\item [(i)] the process $\u\in\mathrm{L}^2(\Omega;\mathrm{L}^{\infty}(0,T;\H)\cap\mathrm{L}^2(0,T;\V))\cap\mathrm{L}^{r+1}(\Omega;\mathrm{L}^{r+1}(0,T;\widetilde{\L}^{r+1}))$ and $\u(\cdot)$ has a $\V\cap\widetilde{\L}^{r+1}$-valued  modification, which is progressively measurable with continuous paths in $\H$ and $\u\in\C([0,T];\H)\cap\mathrm{L}^2(0,T;\V)\cap\mathrm{L}^{r+1}(0,T;\widetilde{\L}^{r+1})$, $\mathbb{P}$-a.s.,
		\item [(ii)] the following equality holds for every $t\in [0, T ]$, as an element of $\V'+\wi\L^{\frac{r+1}{r}},$ $\mathbb{P}$-a.s.
		\begin{align}\label{4.4}
		\u(t)&=\u_0-\int_0^t\left[\mu \A\u(s)+\B(\u(s))+\beta\mathcal{C}(\u(s))\right]\d s+\int_0^t\sigma\d \W(s),
		\end{align}
			\item [(iii)] the following It\^o formula (energy equality) holds true: 
				\begin{align}
			&	\|\u(t)\|_{\H}^2+2\mu \int_0^t\|\u(s)\|_{\V}^2\d s+2\beta\int_0^t\|\u(s)\|_{\widetilde{\L}^{r+1}}^{r+1}\d s\nonumber\\&=\|{\u_0}\|_{\H}^2+\Tr(\sigma\sigma^*)t+2\int_0^t(\sigma\d\W(s),\u(s)),
			\end{align}
			for all $t\in[0,T]$, $\mathbb{P}$-a.s.
	\end{enumerate}
\end{definition}
An alternative version of condition (\ref{4.4}) is to require that for any  $\v\in\V\cap\widetilde{\L}^{r+1}$:
\begin{align}\label{4.5}
(\u(t),\v)&=(\u_0,\v)-\int_0^t\langle\mu \A\u(s)+\B(\u(s))+\beta\mathcal{C}(\u(s)),\v\rangle\d s+\int_0^t\left(\sigma\d \W(s),\v\right),\ \mathbb{P}\text{-a.s.}
\end{align}	
\begin{definition}
	A strong solution $\u(\cdot)$ to (\ref{32}) is called a
	\emph{pathwise  unique strong solution} if
	$\widetilde{\u}(\cdot)$ is an another strong
	solution, then $$\mathbb{P}\Big\{\omega\in\Omega:\u(t)=\widetilde{\u}(t),\ \text{ for all }\ t\in[0,T]\Big\}=1.$$ 
\end{definition}
\begin{theorem}[Theorem 3.7, \cite{MTM8}]\label{exis}
	Let $\x\in\H$, for $r\geq 3$ be given ($2\beta\mu\geq 1,$ for $n=r=3$).  Then there exists a \emph{pathwise unique strong solution}
	$\u(\cdot)$ to the system (\ref{32}) such that \begin{align*}\u&\in\mathrm{L}^2(\Omega;\mathrm{L}^{\infty}(0,T;\H)\cap\mathrm{L}^2(0,T;\V))\cap\mathrm{L}^{r+1}(\Omega;\mathrm{L}^{r+1}(0,T;\widetilde{\L}^{r+1})),\end{align*} with $\mathbb{P}$-a.s., continuous trajectories in $\H$ satisfying 
	\begin{align}
	\E\left[\sup_{t\in[0,T]}\|\u(t)\|_{\H}^2+2\mu \int_0^T\|\u(t)\|_{\V}^2\d t+2\beta\int_0^T\|\u(t)\|_{\widetilde{\L}^{r+1}}^{r+1}\d t\right]\leq C(\|\x\|_{\H}^2+\Tr(\sigma\sigma^*)T).
	\end{align}
\end{theorem}
\begin{remark}[\cite{ICAM,MTM8}]
	For $n=2$ and $r\in[1,3]$,  we know  that $\V\cap\wi\L^{r+1}=\V$ and using Gagliardo-Nirenberg interpolation inequality, we get  $\C([0,T];\H)\cap\mathrm{L}^2(0,T;\V)\subset\mathrm{L}^{r+1}(0,T;\widetilde{\L}^{r+1})$. In this case, the regularity of $\u(\cdot)$ given in Theorem \ref{exis} becomes 
	\begin{align*}
\u\in\mathrm{L}^4(\Omega;\mathrm{L}^{\infty}(0,T;\H))\cap\mathrm{L}^2(\Omega;\mathrm{L}^2(0,T;\V))\cap\mathrm{L}^{r+1}(\Omega;\mathrm{L}^{r+1}(0,T;\widetilde{\L}^{r+1})),
	\end{align*}
	with a $\V\cap\wi\L^{r+1}$-valued modification having paths in $\u\in\C([0,T];\H)\cap\mathrm{L}^2(0,T;\V)$, $\mathbb{P}$-a.s., satisfying 
	\begin{align}
&	\E\left[\sup_{t\in[0,T]}\|\u(t)\|_{\H}^4+4\mu\int_0^T\|\u(t)\|_{\H}^2\|\u(t)\|_{\V}^2\d t+4\beta\int_0^T\|\u(t)\|_{\H}^2\|\u(t)\|_{\widetilde{\L}^{r+1}}^{r+1}\d t\right]\nonumber\\&\leq C(\|\x\|_{\H}^4+\Tr(\sigma\sigma^*)^2T^2).
	\end{align}
\end{remark}

	\section{Asymptotic log-Harnack inequality }\label{sec4}\setcounter{equation}{0}
	In this section, we establish the asymptotic log-Harnack inequality for the trasition semigroup associated with the stochastic convective Brinkman-Forchheimer equations \eqref{32} and discuss about its consequences.  Asymptotic log-Harnack inequality  for stochastic 2D hydrodynamical-type systems with degenerate noise is obtained in \cite{WHSL} and we mainly follow this work to obtain our main results.

For any $\x,\y\in\H$ and $f:\H\to\mathbb{R}$, we denote $$\|\nabla f(\x)\|_{\H}:=\limsup_{\|\x-\y\|_{\H}\to 0}\frac{|f(\x)-f(\y)|}{\|\x-\y\|_{\H}}, \ \text{ and } \ \|\nabla\f\|_{\infty}:=\sup_{\x\in\H}\|\nabla f(\x)\|_{\H}.$$ We define $$\mathrm{Lip}(\H):=\left\{f:\H\to\R:\|\nabla f\|_{\infty}<\infty\right\},$$ as the set of all Lipschitz continuous functions on $\H$. We denote $\u(t,\x)$ as the unique strong solution of \eqref{32} with the initial data $\x\in\H$, and the associated Markov semigroup as $$\mathrm{P}_tf(\x):=\E[f(\u(t,\x))], \ t\geq 0,$$ where $f$ is a bounded measurable function on $\H$. We denote $\mathscr{B}_b(\H):=\{f:\H\to\R:f \text{ is bounded measurable}\}$  and $\mathscr{B}_b^+(\H):=\{f\in\mathscr{B}_b(\H):f\geq 0\}$ (see \cite{GDJZ}). Let us now give the definition of asymptotically strong Feller semigroup introduced Hairer and Mattingly in  \cite{MHJC} (see Definition 3.8). We denote $\mathcal{U}_{\x}$ as the collection of all open sets $\mathrm{U}$ containing $\x$ and indicator function as $\chi_{\A}(\cdot)$, for any measurable set $\A\subset\mathbb{H}$. 
\begin{definition}
	A Markov transition semigroup $\mathrm{P}_t$ on a Polish space $\mathbb{X}$ is \emph{asymptotically strong Feller} at $\x\in\mathbb{X}$, if there exists a totally separating system of pseudo-metrics $\{d_n\}_{n\geq 1}$ for $\mathbb{X}$ and a sequence $t_n > 0$  such that $$\inf_{\mathrm{U}\in\mathcal{U}_{\x}}\limsup_{n\to\infty}\sup_{\y\in\mathrm{U}}d_n(\mathrm{P}_{t_n}(\x,\cdot),\mathrm{P}_{t_n}(\y,\cdot))=0,$$ where $\mathrm{P}_t(\x,\A)=\mathrm{P}_t\chi_{\A}(\x)$, for any $\x\in\mathbb{X}$ and measurable set $\A\subset\mathbb{X}$. The semigroup $\mathrm{P}_t$ is called asymptotically strong Feller if it is asymptotically strong Feller at any $\x \in\mathbb{X}$.
\end{definition}
For the definition of totally separating system of pseudo-metrics, interested readers are referred to see \cite{MHJC}. 
From Definition 2.2, \cite{JBFY}, we know that the following inequality is called an asymptotic log-Harnack inequality for the transition semigroup $\mathrm{P}_t$: 
$$\mathrm{P}_t\log f(\y)\leq \mathrm{P}_t\log f(\x)+\Phi(\x,\y)+\Psi_t(\x,\y)\|\nabla \log f\|_{\infty}, \ t>0,$$ for any $f\in\mathscr{B}_b^+(\H)$ with $\|\nabla\log f\|_{\infty}<\infty$, where $\Phi,\Psi_t:\H\times\H\to(0,\infty)$  are measurable with $\Psi_t\downarrow 0$ as $t\uparrow\infty$. 

\subsection{The case $n=2$ and $r\in[1,3]$} 
Let us now state and prove our main result on the asymptotic log-Harnack inequality for the transition semigroup $\mathrm{P}_t$ associated with the SCBF equations \eqref{32} for the case $n=2$ and $r\in[1,3]$. 
\begin{theorem}\label{maint}
Let $n=2$ and $r\in[1,3]$ and $\u(t,\x)$ be the unique strong solution to the system \eqref{32}	with the initial data $\x\in\H$. Suppose that the noise co-efficient $\sigma$ satisfies Assumption \ref{ass4.1} and $$\lambda_1\mu^3\geq 8\Tr(\sigma\sigma^*).$$ Then, for any $\x,\y\in\H$ and for any $f\in\mathscr{B}^+_b(\H)$ with $\|\nabla\log f\|_{\infty}<\infty$, we have the following \emph{asymptotic log-Harnack inequality}:
	\begin{align}\label{4.2}
	\mathrm{P}_t\log f(\y)&\leq \log\mathrm{P}_tf(\x)+\gamma e^{k\|\y\|_{\H}^2}\|\x-\y\|_{\H}^2+2e^{-\theta t}e^{k\|\y\|_{\H}^2}\|\x-\y\|_{\H}\|\nabla\log f\|_{\infty}, \ t >0, 
	\end{align}
	where the constants $k\leq \frac{\lambda_1\mu}{4\Tr(\sigma\sigma^*)}$, $\theta=\frac{\mu\lambda_{N_0}-k\Tr(\sigma\sigma^*)}{2}$ and $\gamma=\frac{\mu^2C_{\sigma}^2\lambda_{N_0}^2}{4(\mu\lambda_{N_0}-k\Tr(\sigma\sigma^*))}.$ 
\end{theorem}
Applying Theorem 2.1, \cite{JBFY}, similar results obtained in Corollary 3.1, \cite{WHSL}, we have the following corollary. 
\begin{corollary}\label{cor4.4} 
	Under the assumptions of Theorem \ref{maint}, we have 
	\begin{enumerate}
		\item [(1)] \textbf{Gradient estimate.} For any $\y\in\H$, $t>0$ and $f\in\mathrm{Lip}_b(\H):=\mathrm{Lip}(\H)\cap\mathscr{B}_b(\H),$ we have 
		\begin{align}
		\|\nabla\mathrm{P}_tf(\y)\|_{\H}\leq \sqrt{2\gamma e^{k\|\y\|_{\H}^2}}\sqrt{\mathrm{P}_tf^2(\y)-(\mathrm{P}_tf)^2(\y)}+2e^{-\theta t+k\|\y\|_{\H}^2}\|\nabla f\|_{\infty}.
		\end{align}
		In particular, $\mathrm{P}_t$ is asymptotically strong Feller. 
		\item [(2)] \textbf{Asymptotic irreducibility.} Let $\y\in\H$ and $\A\subset\H$ be a measurable set such that $$\delta(\y,\A):= \liminf_{t\to\infty}\mathrm{P}_t(\y,\A)>0.$$ Then, we have $$\liminf_{t\to\infty}\mathrm{P}_t(\x,\A_{\e})>0, \ \text{ for any } \ \x\in\H, \ \e>0,$$ where $\A_{\e}:=\{\x\in\H:\rho(\x,\A)<\e\}$ with $\rho(\x,\A):=\inf\limits_{\z\in\A}\|\x-\z\|_{\H}.$ Furthermore, for any $\e_0\in(0,\delta(\y,\A))$, there exists a constant $t_0>0$ such that  $$\mathrm{P}_t(\x,\A_{\e})>0,\ \text{ provided }\  t\geq t_0\ \text{ and }\ 2e^{-\theta t+k\|\y\|_{\H}^2}\|\x-\y\|_{\H}<\e\e_0. $$ 
		\item [(3)] \textbf{Asymptotic heat kernel estimate.} If  $\mathrm{P}_t$ has an invariant probability measure $\nu$,  then for any $f\in\mathscr{B}_b^+(\H)$ with $\|\nabla f\|_{\infty}<\infty$, we have 
		$$\limsup_{t\to\infty}\mathrm{P}_tf(\y)\leq\log\left(\frac{\nu(e^{f})}{\int_{\H}e^{-\gamma e^{k\|\y\|_{\H}^2\|\x-\y\|_{\H}^2}}\nu(\d \x)}\right), \ \y\in\H. $$ Consequently, for any closed set $\A\subset\H$ with $\nu(\A)=0$, we have $$\lim_{t\to\infty}\mathrm{P}_t\chi_{\A}(\y)=0,\ \text{ for all }\ \y\in\H.$$
			\item [(4)] \textbf{Uniqueness of invariant probability measure.} The Markovian transition semigroup $\mathrm{P}_t$ has at most one invariant probability measure. 
	\end{enumerate}
\end{corollary}
As discussed in \cite{WHSL}, we prove Theorem \ref{maint} and Corollary \ref{cor4.4} via coupling method by change of measures in infinite dimensional spaces. 
Given any $\y\in\H$, let us construct an auxiliary process $\v(\cdot)$ by the following equation:
\begin{equation}\label{4p4}
\left\{
\begin{aligned}
\d\v(t)+[\mu\A\v(t)+\B(\v(t))+\beta\mathcal{C}(\v(t))]\d t&=\sigma\d\W(t)+\frac{\mu\lambda_{N_0}}{2}(\u(t)-\v(t))^l\d t, \\
\v(0)&=\y. 
\end{aligned}\right.
\end{equation}
The system \eqref{4p4} can be interpreted as the system \eqref{32} with an additional control term. Or in other words, one can treat the system \eqref{4p4} as an analogue of the  system \eqref{32} with the operator $\A$ replaced by $\widetilde{\A}:=\A+\frac{\mu\lambda_{N_0}}{2}\mathrm{P}_{N_0}$, and an additional term $-\frac{\mu\lambda_{N_0}}{2}\mathrm{P}_{N_0}\u$. Using similar arguments as in Remark 8, \cite{AKMS}, we can use  Girsanov's theorem to prove that the strong solution of the system \eqref{4p4} is uniquely defined. Furthermore, one can write the system \eqref{4p4} in the form \eqref{32} with $\x$ changed to $\y$ and $\W(t)$ changed to $$\wi\W(t):=\W(t)+\int_0^th(s)\d s,\ h(t):=\frac{\mu\lambda_{N_0}}{2}\sigma^{-1}(\u(t)-\v(t))^l, \ t\geq 0,$$ where $\sigma$ is defined in the Assumption \ref{ass4.1}. Let us now define 
\begin{align}
\Phi(t):=\exp\left\{-\int_0^t(h(s),\d \W(s))-\frac{1}{2}\int_0^t\|h(s)\|_{\H}^2\d s\right\}. 
\end{align}
Using Girsanov's theorem, we conclude that $\{\wi\W(t)\}_{t\geq 0}$ is a Wiener process on $\H$ under the weighted probability measure $\widetilde{\mathbb{P}}$ and is uniquely defined by 
\begin{align}
\frac{\d\widetilde{\mathbb{P}}}{\d\mathbb{P}}\Big|_{\mathscr{F}_t}:=\Phi(t),\  t\geq 0. 
\end{align}
Next, we rewrite the systems \eqref{32} and \eqref{4p4} as 
\begin{equation}\label{4.7}
\left\{
\begin{aligned}
\d\u(t)+[\mu \A\u(t)+\B(\u(t))+\beta\mathcal{C}(\u(t))]\d t&=\sigma\d\W(t), \\
\d\v(t)+[\mu\A\v(t)+\B(\v(t))+\beta\mathcal{C}(\v(t))]\d t&=\sigma\d\W(t)+\frac{\mu\lambda_{N_0}}{2}(\u(t)-\v(t))^l\d t,
\end{aligned}
\right. 
\end{equation}
with initial values $\u(0)=\x$ and $\v(0)=\y$, respectively. With the above setting, we have the following estimate:
\begin{lemma}\label{lem4.6}
For any $k\leq\frac{\lambda_1\mu}{4\Tr(\sigma\sigma^*)}$, we have 
	\begin{align}\label{4.8}
&	\E\left\{\exp\left[k\sup_{t\geq 0}\left(\|\u(t)\|_{\H}^2+\mu\int_0^t\|\u(s)\|_{\V}^2\d s+2\beta\int_0^t\|\u(s)\|_{\wi\L^{r+1}}^{r+1}\d s-\Tr(\sigma\sigma^*)t\right)\right]\right\}\nonumber\\&\quad\leq 2\exp\left(k\|\x\|_{\H}^2\right),
\end{align} and \begin{align}
&	\E_{\widetilde{\mathbb{P}}}\left\{\exp\left[k\sup_{t\geq 0}\left(\|\v(t)\|_{\H}^2+\mu\int_0^t\|\v(s)\|_{\V}^2\d s+2\beta\int_0^t\|\v(s)\|_{\wi\L^{r+1}}^{r+1}\d s-\Tr(\sigma\sigma^*)t\right)\right]\right\}\nonumber\\&\quad\leq 2\exp\left(k\|\y\|_{\H}^2\right).\label{4.9}
	\end{align}
\end{lemma}
\begin{proof}
	Let us define 
	\begin{align}
	\M(t):=2\int_0^t(\sigma\d\W(s),\u(s)). 
	\end{align}
	Then, it can be easily seen that $\M(t)$ is a continuous martingale and a bound for  its quadratic variation process is given by 
	\begin{align}\label{4.11}
	\langle\M\rangle(t)&=4\int_0^t\Tr(\sigma\sigma^*\u(s)\otimes\u(s))\d s\leq 4\Tr(\sigma\sigma^*)\int_0^t\|\u(s)\|_{\H}^2\d s\nonumber\\&\leq\frac{4}{\lambda_1}\Tr(\sigma\sigma^*)\int_0^t\|\u(s)\|_{\V}^2\d s,
	\end{align}
	where we used Poincar\'e's inequality. Let us set $k_0:=\frac{\lambda_1\mu}{2\Tr(\sigma\sigma^*)}$ and $\mathcal{M}_{k_0}(t):=\M(t)-\frac{k_0}{2}\langle\M\rangle(t)$. An application of the infinite dimensional It\^o formula to the process $\|\u(\cdot)\|_{\H}^2$ (see \cite{MTM8}) yields 
	\begin{align}\label{4.12}
&	\|\u(t)\|_{\H}^2+2\mu\int_0^t\|\u(s)\|_{\V}^2\d s+2\beta\int_0^t\|\u(s)\|_{\wi\L^{r+1}}^{r+1}\d s\nonumber\\&=\|\x\|_{\H}^2+\Tr(\sigma\sigma^*)t+2\int_0^t(\u(s),\sigma\d\W(s)), 
	\end{align}
	for all $t\in[0,T]$, $\mathbb{P}$-a.s. From \eqref{4.12}, we infer that 
	\begin{align}\label{4.13}
	&	\|\u(t)\|_{\H}^2+\mu\int_0^t\|\u(s)\|_{\V}^2\d s+2\beta\int_0^t\|\u(s)\|_{\wi\L^{r+1}}^{r+1}\d s\nonumber\\&=\|\x\|_{\H}^2-\mu\int_0^t\|\u(s)\|_{\V}^2\d s+\Tr(\sigma\sigma^*)t+\M(t)\nonumber\\&=\|\x\|_{\H}^2+\Tr(\sigma\sigma^*)t+\mathcal{M}_{k_0}(t)+\frac{k_0}{2}\langle\M\rangle(t)-\mu\int_0^t\|\u(s)\|_{\V}^2\d s\nonumber\\&\leq \|\x\|_{\H}^2+\Tr(\sigma\sigma^*)t+\mathcal{M}_{k_0}(t)+\left(\frac{2k_0}{\lambda_1}\Tr(\sigma\sigma^*)-\mu\right)\int_0^t\|\u(s)\|_{\V}^2\d s\nonumber\\&\leq \|\x\|_{\H}^2+\Tr(\sigma\sigma^*)t+\mathcal{M}_{k_0}(t),
	\end{align}
	where we used \eqref{4.11}. Note that $e^{k_0\mathcal{M}_{k_0}(\cdot)}$ is a supermartingale and using the maximal supermartingale inequality, we get 
	\begin{align}\label{4.14}
	\mathbb{P}\left\{\sup_{t\geq 0}\mathcal{M}_{k_0}(t)\geq y \right\}=\mathbb{P}\left\{\sup_{t\geq 0}\exp\left(k_0\mathcal{M}_{k_0}(t)\right)\geq e^{k_0y}\right\}\leq e^{-k_0y}\E\left\{\exp\left(k_0\mathcal{M}_{k_0}(T)\right)\right\}=e^{-k_0y}.
	\end{align}
	For $k\leq\frac{k_0}{2}$, the above inequality implies  
	\begin{align}\label{4.15}
	\E\left\{\exp\left(k\sup_{t\geq 0}\mathcal{M}_{k_0}(t)\right)\right\}=1+k\int_0^{\infty}e^{ky}\mathbb{P}\left\{\sup_{t\geq 0}\mathcal{M}_{k_0}(t)\geq y\right\}\d y\leq 2,
	\end{align}
	using \eqref{4.14}. Thus, from \eqref{4.13}, it is immediate that 
	\begin{align}
	&	\E\left\{\exp\left[k\sup_{t\geq 0}\left(\|\u(t)\|_{\H}^2+\mu\int_0^t\|\u(s)\|_{\V}^2\d s+2\beta\int_0^t\|\u(s)\|_{\wi\L^{r+1}}^{r+1}\d s\right)\right]\right\}\nonumber\\&\leq \exp\left(k\|\x\|_{\H}^2\right)\exp\left(k\sup_{t\geq 0}\Tr(\sigma\sigma^*)t\right)\E\left\{\exp\left(k\sup_{t\geq 0}\mathcal{M}_{k_0}(t)\right)\right\},
	\end{align}
	and the inequality \eqref{4.8} follows by using \eqref{4.15}. The inequality \eqref{4.9} can be obtained in a similar way. 
\end{proof}
\begin{lemma}\label{lem4p6}
Let us 	assume that $\lambda_1\mu^3\geq 8\Tr(\sigma\sigma^*)$. Then, there exists a constant $k\leq \frac{\lambda_1\mu}{4\Tr(\sigma\sigma^*)}$ such that the asymptotic coupling $(\u(\cdot),\v(\cdot))$ given in \eqref{4.7} satisfies: 
\begin{align}\label{4.18}
\mathbb{E}_{\widetilde{\mathbb{P}}}\left[\|\u(t)-\v(t)\|_{\H}^2\right]\leq 2e^{-\theta t}e^{k\|\y\|_{\H}^2}\|\x-\y\|_{\H}^2, \ t\geq 0 ,
\end{align}
where $\theta=\frac{1}{2}(\mu\lambda_{N_0}-k\Tr(\sigma\sigma^*))$. 
\end{lemma}
\begin{proof}
	Note that $\w(\cdot):=\u(\cdot)-\v(\cdot)$ satisfies the following It\^o stochastic differential: 
	\begin{equation}\label{4.19}
	\left\{
	\begin{aligned}
	\d\w(t)+[\mu\w(t)+\B(\u(t))-\B(\v(t))+\beta(\mathcal{C}(\u(t))-\mathcal{C}(\v(t)))]\d t&=-\frac{\mu\lambda_{N_0}}{2}(\w(t))^l\d t,\\
	\w(0)&=\x-\y.
	\end{aligned}
	\right. 
	\end{equation}
Taking inner product with $\w(\cdot)$ to the first equation in \eqref{4.19}, we find 
	\begin{align}\label{4.20}
&\frac{\d}{\d t}	\|\w(t)\|_{\H}^2+2\mu\|\w(t)\|_{\V}^2\nonumber\\&=-2\langle\B(\u(t))-\B(\v(t)),\w(t)\rangle-2\beta\langle\mathcal{C}(\u(t))-\mathcal{C}(\v(t)),\w(t)\rangle-\mu \lambda_{N_0}(\mathrm{P}_{N_0}\w(t),\w(t)), 
	\end{align}
	for a.e. $t\in[0,T]$. 
	Using Ladyzhenskaya's, H\"older's and Young's inequalities, we estimate $-2\langle\B(\u)-\B(\v),\w\rangle$ as 
	\begin{align}\label{4.22}
	-2\langle\B(\u)-\B(\v),\w\rangle&=-2\langle\B(\w,\v),\w\rangle \leq 2\|\v\|_{\V}\|\w\|_{\wi\L^4}^2\leq 2\sqrt{2}\|\v\|_{\V}\|\w\|_{\H}\|\w\|_{\V}\nonumber\\&\leq{\mu}\|\w\|_{\V}^2+\frac{2}{\mu}\|\v\|_{\V}^2\|\w\|_{\H}^2. 
	\end{align}
	Using \eqref{214}, we have 
	\begin{align}\label{4.21}
	-2\beta\langle\mathcal{C}(\u)-\mathcal{C}(\v),\w\rangle&\leq -\frac{\beta}{2^{r-2}}\|\w\|_{\wi\L^{r+1}}^{r+1}, \ \text{ for } \ r\geq 1. 
	\end{align}
	Substituting \eqref{4.22} and \eqref{4.21} in \eqref{4.20}, we obtain 
	\begin{align}\label{4.23}
&\frac{\d}{\d t}	\|\w(t)\|_{\H}^2+\frac{\beta}{2^{r-2}}\|\w(t)\|_{\wi\L^{r+1}}^{r+1}\nonumber\\&\leq\frac{2}{\mu}\|\v(t)\|_{\V}^2\|\w(t)\|_{\H}^2-\left(\mu\|\w(t)\|_{\V}^2+\mu\lambda_{N_0}\|\mathrm{P}_{N_0}\w(t)\|_{\H}^2\right). 
	\end{align}
	Using \eqref{4.3}, we estimate the final term from the right hand side of the inequality \eqref{4.23} as 
	\begin{align}\label{4.24} 
	\mu\|\w\|_{\V}^2+\mu\lambda_{N_0}\|\mathrm{P}_{N_0}\w\|_{\H}^2&\geq \mu\|(\mathrm{I}-\mathrm{P}_{N_0})\w\|_{\V}^2+\mu\lambda_{N_0}\|\mathrm{P}_{N_0}\w\|_{\H}^2\nonumber\\&\geq \mu\lambda_{N_0}\|(\mathrm{I}-\mathrm{P}_{N_0})\w\|_{\H}^2+\mu\lambda_{N_0}\|\mathrm{P}_{N_0}\w\|_{\H}^2\nonumber\\&=\mu\lambda_{N_0}\|\w\|_{\H}^2.
	\end{align}
	Applying \eqref{4.24} in \eqref{4.23}, we obtain 
		\begin{align}\label{4.25}
	&\frac{\d}{\d t}	\|\w(t)\|_{\H}^2+\mu\lambda_{N_0}\|\w(t)\|_{\H}^2\leq\frac{2}{\mu}\|\v(t)\|_{\V}^2\|\w(t)\|_{\H}^2 .
	\end{align}
	An application of Gronwall's inequality yields 
	\begin{align}
	\|\w(t)\|_{\H}^2\leq \|\w(0)\|_{\H}^2e^{-\mu\lambda_{N_0}t}\exp\left(\frac{2}{\mu}\int_0^t\|\v(s)\|_{\V}^2\d s\right), 
	\end{align}
	for all $t\in[0,T]$. Since $\lambda_1\mu^3\geq 8\Tr(\sigma\sigma^*)$, there exists a constant $k\leq \frac{\lambda_1\mu}{4\Tr(\sigma\sigma^*)}$ such that 
	\begin{align}
	\mathbb{E}_{\widetilde{\mathbb{P}}}\|\w(t)\|_{\H}^2&\leq\|\x-\y\|_{\H}^2e^{-\mu\lambda_{N_0}t}	\mathbb{E}_{\widetilde{\mathbb{P}}}\left[\exp\left(\frac{2}{\mu}\int_0^t\|\v(s)\|_{\V}^2\d s\right)\right]\nonumber\\&\leq 2\|\x-\y\|_{\H}^2e^{k\|\y\|_{\H}^2}\exp\left\{-\left(\mu\lambda_{N_0}-k\Tr(\sigma\sigma^*)\right)t\right\},
	\end{align}
	where we used Lemma \ref{lem4.6} (see \eqref{4.9}). 
	\iffalse 
	On the other hand from \eqref{4.25}, we get 
	\begin{align}
	\frac{\d}{\d t}	\|\w(t)\|_{\H}+\frac{\mu\lambda_{N_0}}{2}\|\w(t)\|_{\H}\leq\frac{1}{\mu}\|\v(t)\|_{\V}^2\|\w(t)\|_{\H},
	\end{align}
for a.e. $t\in[0,T]$. 	Applying Gronwall's inequality and proceeding similarly as before, we obtain
	\begin{align}
		\mathbb{E}_{\widetilde{\mathbb{P}}}\|\w(t)\|_{\H}\leq 2\|\x-\y\|_{\H}e^{k\|\y\|_{\H}^2}\exp\left\{-\left(\frac{\mu\lambda_{N_0}}{2}-k\Tr(\sigma\sigma^*)\right)t\right\},
	\end{align}
	\fi 
	It should be noted that $\{\lambda_k\}_{k=1}^{\infty}$ is  an unbounded increasing sequence,  and thus we can choose $N_0$ large enough so that $\mu\lambda_{N_0}>k\Tr(\sigma\sigma^*)$. Hence, the estimate \eqref{4.18} follows. 
\end{proof}

In order to prove our main Theorem (see Theorem \ref{maint}), we need the following inequality (see Lemma 2.4, \cite{MAAT}). 
\begin{lemma}
	Let $f\geq 0$ with $\E[f]>0$. Then, for any measurable function $g$, we have 
	\begin{align}\label{4.42}
	\E[fg]\leq \E[f]\log\E[e^g]+\E[f\log f]-\E[f]\log\E[f]. 
	\end{align}
\end{lemma}
\begin{proof}[Proof of Theorem \ref{maint}]
Using the uniqueness of strong solution to the system \eqref{4p4}, we have 
\begin{align}
\mathrm{P}_t[f(\y)]=\E_{\widetilde{\mathbb{P}}}\left[f(\v(t,\y))\right], \ t\geq 0, \ \y\in\H,
\end{align}
for all $f\in\mathscr{B}_b(\H)$. Thus, for any $f\in\mathscr{B}_b^+(\H)$ with $\|\nabla\log f\|_{\infty}<\infty$, by using the definition of $\|\nabla\log f\|_{\infty}$, \eqref{4.42} and \eqref{4.18}, we get 
\begin{align}\label{4.44}
\mathrm{P}_t[\log f(\y)]&=\E_{\widetilde{\mathbb{P}}}\left[\log f(\v(t,\y))\right]\nonumber\\&= \E_{\widetilde{\mathbb{P}}}\left[\log f(\u(t,\x))\right]+\E_{\widetilde{\mathbb{P}}}\left[\log f(\v(t,\y))-\log f(\u(t,\x))\right]\nonumber\\&\leq \E\left[\Phi(t)\log f(\u(t,\x))\right]+\|\nabla\log f\|_{\infty}\E_{\widetilde{\mathbb{P}}}\left[\|\u(t,\x)-\v(t,\y)\|_{\H}\right]\nonumber\\&\leq \E\left[\Phi(t)\log f(\u(t,\x))\right]+\|\nabla\log f\|_{\infty}\E_{\widetilde{\mathbb{P}}}\left[\|\u(t,\x)-\v(t,\y)\|_{\H}^2\right]^{1/2}\nonumber\\&\leq\E\left[\Phi(t)\log\Phi(t)\right]+\log\mathrm{P}_tf(\x)+2e^{-\theta t}e^{k\|\y\|_{\H}^2}\|\x-\y\|_{\H}\|\nabla\log f\|_{\infty}.
\end{align}
Using Fubini's theorem, \eqref{4.1} and \eqref{4.18}, we estimate $\E\left[\Phi(t)\log\Phi(t)\right]$ as 
\begin{align}\label{4.45}
\E\left[\Phi(t)\log\Phi(t)\right]&=\E_{\widetilde{\mathbb{P}}}\left[\log\Phi(t)\right]=\frac{\mu^2\lambda_{N_0}^2}{8}\E_{\widetilde{\mathbb{P}}}\left[\int_0^t\|\sigma^{-1}(\u(s)-\v(s))^l\|_{\H}^2\d s\right]\nonumber\\&\leq\frac{C_{\sigma}^2\mu^2\lambda_{N_0}^2}{8}\int_0^t\E_{\widetilde{\mathbb{P}}}\left[\|\u(s)-\v(s)\|_{\H}^2\right]\d s\nonumber\\&\leq \frac{C_{\sigma}^2\mu^2\lambda_{N_0}^2}{4(\mu\lambda_{N_0}-k\Tr(\sigma\sigma^*))} e^{k\|\y\|_{\H}^2}\|\x-\y\|_{\H}^2. 
\end{align}
Hence \eqref{4.2} follows by applying \eqref{4.44} in \eqref{4.45}. 
\end{proof}

\begin{proof}[Proof of Corollary \ref{cor4.4}]
Let us define 
\begin{align*}\Theta(\x,\y)=\gamma e^{k\|\y\|_{\H}^2}\|\x-\y\|_{\H}^2\ \text{  and }\ \Psi_t(\x,\y)=2e^{-\theta t}e^{k\|\y\|_{\H}^2}\|\x-\y\|_{\H}. \end{align*}	From Theorem 2.1, \cite{JBFY}, we know that if $\mathrm{P}_t$ satisfies \eqref{4.2} and  $\Theta,\Psi_t:\H\times\H\to\R^+$ are symmetric with $\Psi_t\downarrow 0$ as $t\to\infty$, then we arrive at the assertions given in Corollary \ref{cor4.4}, provided we prove  the  existence of an invariant measure for  the transition semigroup $\{\mathrm{P}_t\}_{t\geq 0}$. We use the standard Krylov-Bogoliubov approach to obtain the existence of an invariant measure (see \cite{DaZ}).  Let us define $$\nu_n:=\frac{1}{n}\int_0^n\delta_\mathbf{0}\mathrm{P}_t\d t, \ n\geq 1,$$ where $\delta_\mathbf{0}$ is a Dirac measure at $\mathbf{0}\in\H$. It is clear that each $\mu_n$ is a probability measure. Let $\u(t,\x)$ be the unique strong solution of the system  \eqref{32} with the initial data $\x\in\H$. Then, a calculation similar to \eqref{4.23} yields
	\begin{align}\label{4.47}
&\frac{\d}{\d t}	\|\u(t,\x)-\u(t,\y)\|_{\H}^2+\mu\|\u(t,\x)-\u(t,\y)\|_{\V}^2+\frac{\beta}{2^{r-2}}\|\u(t,\x)-\u(t,\y)\|_{\wi\L^{r+1}}^{r+1}\nonumber\\&\leq\frac{2}{\mu}\|\u(t,\x)\|_{\V}^2\|\u(t,\x)-\u(t,\y)\|_{\H}^2. 
\end{align}
An application of Gronwall's inequality in \eqref{4.47} yields 
\begin{align}\label{4.48}
\|\u(t,\x)-\u(t,\y)\|_{\H}^2\leq\|\x-\y\|_{\H}^2\exp\left(\frac{2}{\mu}\int_0^t\|\u(s,\x)\|_{\V}^2\d s\right). 
\end{align}
From Theorem \ref{exis}, we infer that $\E\left(\int_0^t\|\u(s,\x)\|_{\V}^2\d s\right)<\infty,$ which together with \eqref{4.48} implies that the transition semigroup $\{\mathrm{P}_t\}_{t\geq 0}$ is a Feller semigroup. Thus, in order to prove the existence of an invariant measure for $\{\mathrm{P}_t\}_{t\geq 0}$, we only need to prove the tightness of the measure $\{\nu_n:n\geq 1\}$. Applying the infinite dimensional It\^o formula to the process $\|\u(\cdot,\x)\|_{\H}^2$ and then taking expectation, we obtain 
\begin{align}
\E\left[\|\u(t)\|_{\H}^2+2\mu\int_0^t\|\u(s)\|_{\V}^2\d s+2\beta\int_0^t\|\u(s)\|_{\wi\L^{r+1}}^{r+1}\d s\right]=\|\x\|_{\H}^2+\Tr(\sigma\sigma^*)t,
\end{align}
for all $t\in[0,T]$. Thus, it is immediate that 
\begin{align}\label{4.50}
\nu_n(\|\cdot\|_{\V}^2)=\frac{1}{n}\int_0^n\E\left[\|\u(t,\mathbf{0})\|_{\V}^2\right]\d t=\frac{\Tr(\sigma\sigma^*)}{2\mu}. 
\end{align}
Note also that $\nu_n(\|\cdot\|_{\wi\L^{r+1}}^{r+1})=\frac{\Tr(\sigma\sigma^*)}{2\beta}. $
We know that the embedding of $\V\subset\H$ is compact and hence for any $K\in(0,\infty)$, the set $\{\x\in\H:\|\x\|_{\V}\leq K\}$ is relatively compact in $\H$. Thus, the inequality \eqref{4.50} implies that  $\{\mu_n\}_{n\geq 1}$ is tight. Hence, using Prohorov's theorem there exists a probability measure $\nu$ and a subsequence $\{\nu_{n_k}\}_{k\in\mathbb{N}}$ such that $\nu_{n_k} \xrightarrow{w}\nu$, as $k\to\infty$. One can easily show that $\nu$ is an invariant probability measure for $\{\mathrm{P}_t\}_{t\geq 0}$. Thus, by the Krylov-Bogoliubov theorem (or by a  result of Chow and Khasminskii see \cite{CHKH}),  $\nu$ results to be an invariant measure for the transition semigroup $\{\mathrm{P}_t\}_{t\geq 0}$,  defined by 	$\mathrm{P}_t\varphi(\x)=\E\left[\varphi(\u(t,\x))\right],$ for all $\varphi\in\C_b(\H)$, where $\u(\cdot)$ is the unique strong solution of the system (\ref{32}) with the initial condition $\x\in\H$.
\end{proof}

\subsection{The cases $n=2,3$ and $r\in(3,\infty)$} 
In this subsection, we state and prove the results analogous to Theorem \ref{maint} and Corollary \ref{cor4.4} for the cases $n=2,3$ and $r\in(3,\infty)$. The main results of this subsection are the following: 

\begin{theorem}\label{maint1}
	Let $n=2,3$, $r\in(3,\infty)$ and $\u(t,\x)$ be the unique strong solution to the system \eqref{32}	with the initial data $\x\in\H$. Suppose that the noise co-efficient $\sigma$ satisfies the Assumption \ref{ass4.1}.  Then, for any $\x,\y\in\H$ and for any $f\in\mathscr{B}^+_b(\H)$ with $\|\nabla\log f\|_{\infty}<\infty$, we have the following \emph{asymptotic log-Harnack inequality}:
	\begin{align}\label{4p2}
	\mathrm{P}_t\log f(\y)&\leq \log\mathrm{P}_tf(\x)+\widetilde{\gamma}\|\x-\y\|_{\H}^2+e^{-\wi{\theta}t}\|\x-\y\|_{\H}\|\nabla\log f\|_{\infty}, \ t >0, 
	\end{align}
	where the constants $\wi{\theta}=\frac{1}{2}\left(\mu\lambda_{N_0}-\frac{r-3}{\mu(r-1)}\left(\frac{2}{\beta\mu (r-1)}\right)^{\frac{2}{r-3}}\right)$ and $\widetilde{\gamma}=\frac{C_{\sigma}^2\mu^2\lambda_{N_0}^2}{8\left(\mu\lambda_{N_0}-\frac{r-3}{\mu(r-1)}\left(\frac{2}{\beta\mu (r-1)}\right)^{\frac{2}{r-3}}\right)} $. 
	
	For $n=3$, $r=3$ and $\beta\mu>1$, we have the following \emph{asymptotic log-Harnack inequality}:
	\begin{align}\label{4a2}
	\mathrm{P}_t\log f(\y)&\leq \log\mathrm{P}_tf(\x)+ \frac{C_{\sigma}^2\mu\lambda_{N_0}}{8}\|\x-\y\|_{\H}^2+e^{-\frac{\mu\lambda_{N_0}t}{2}}\|\x-\y\|_{\H}\|\nabla\log f\|_{\infty}, \ t >0, 
	\end{align}
\end{theorem}
Analogous to Corollary \ref{cor4.4}, we have the following result for $n=2,3$ and $r\in(3,\infty)$. For the case $n=3$, $r=3$ and $\beta\mu>1$, one has to replace $\wi{\theta}$ with $\frac{\mu\lambda_{N_0}}{2}$ and $\wi{\gamma}$ with $ \frac{C_{\sigma}^2\mu\lambda_{N_0}}{8}$ in the Corollary given below. 
\begin{corollary}\label{cor4p4} 
	Under the assumptions of Theorem \ref{maint}, we have 
	\begin{enumerate}
		\item [(1)] \textbf{Gradient estimate.} For any $\y\in\H$, $t>0$ and $f\in\mathrm{Lip}_b(\H):=\mathrm{Lip}(\H)\cap\mathscr{B}_b(\H),$ we have 
		\begin{align}
		\|\nabla\mathrm{P}_tf(\y)\|_{\H}\leq \sqrt{\wi{\gamma}}\sqrt{\mathrm{P}_tf^2(\y)-(\mathrm{P}_tf)^2(\y)}+e^{-\widetilde{\theta}t}\|\nabla f\|_{\infty}.
		\end{align}
		In particular, $\mathrm{P}_t$ is asymptotically strong Feller. 
		\item [(2)] \textbf{Asymptotic irreducibility.} Let $\y\in\H$ and $\A\subset\H$ be a measurable set such that $$\delta(\y,\A):= \liminf_{t\to\infty}\mathrm{P}_t(\y,\A)>0.$$ Then, we have $$\liminf_{t\to\infty}\mathrm{P}_t(\x,\A_{\e})>0, \ \text{ for any } \ \x\in\H, \ \e>0,$$ where $\A_{\e}:=\{\x\in\H:\rho(\x,\A)<\e\}$ with $\rho(\x,\A):=\inf\limits_{\z\in\A}\|\x-\z\|_{\H}.$ Furthermore, for any $\e_0\in(0,\delta(\y,\A))$, there exists a constant $t_0>0$ such that $$\mathrm{P}_t(\x,\A_{\e})>0, \ \text{ provided }\ t\geq t_0\ \text{ and }\ e^{-\wi{\theta}t}\|\x-\y\|_{\H}<\e\e_0 .$$ 
		\item [(3)] \textbf{Asymptotic heat kernel estimate.} If  $\mathrm{P}_t$ has an invariant probability measure $\nu$,  then for any $f\in\mathscr{B}_b^+(\H)$ with $\|\nabla f\|_{\infty}<\infty$, we have 
		$$\limsup_{t\to\infty}\mathrm{P}_tf(\y)\leq\log\left(\frac{\nu(e^{f})}{\int_{\H}e^{-\wi{\gamma}\|\x-\y\|_{\H}^2}\nu(\d \x)}\right), \ \y\in\H. $$ Consequently, for any closed set $\A\subset\H$ with $\nu(\A)=0$, we have $$\lim_{t\to\infty}\mathrm{P}_t\chi_{\A}(\y)=0,\ \text{ for all }\ \y\in\H.$$
		\item [(4)] \textbf{Uniqueness of invariant probability measure.} The Markovian transition semigroup $\mathrm{P}_t$ has at most one invariant probability measure. 
	\end{enumerate}
\end{corollary}

We begin with the following Lemma, which is analogous to Lemma \ref{lem4p6}.
\begin{lemma}
	For $n=2,3$ and $r\in(3,\infty)$, the asymptotic coupling $(\u(\cdot),\v(\cdot))$ given in \eqref{4.7} satisfies: 
	\begin{align}\label{4p18}
	\mathbb{E}_{\widetilde{\mathbb{P}}}\left[\|\u(t)-\v(t)\|_{\H}^2\right]\leq e^{-\wi{\theta}t}\|\x-\y\|_{\H}^2, \ t\geq 0 ,
	\end{align}
	where  $\wi{\theta}=\left(\mu\lambda_{N_0}-\frac{r-3}{\mu(r-1)}\left(\frac{2}{\beta\mu (r-1)}\right)^{\frac{2}{r-3}}\right)$. 
	
	For $n=3$, $r=3$ and $\beta\mu>1$, we have 
	\begin{align}\label{4a18}
	\mathbb{E}_{\widetilde{\mathbb{P}}}\left[\|\u(t)-\v(t)\|_{\H}^2\right]\leq e^{-\mu\lambda_{N_0}t}\|\x-\y\|_{\H}^2, \ t\geq 0. 
	\end{align}
\end{lemma}

\begin{proof}
	The case of $n=2,3$ and $r\in(3,\infty)$ follows similarly as in the proof of Lemma \ref{lem4p6} except for the estimates \eqref{4.21} and \eqref{4.22}. It is important to note that, in these cases, we are  not using Lemma \ref{lem4.6}.   From \eqref{2.23}, we have 
	\begin{align}\label{4.30}
	-2\beta\langle\mathcal{C}(\u)-\mathcal{C}(\v),\w\rangle&\leq-\beta\||\v|^{\frac{r-1}{2}}\w\|_{\H}^2. 
	\end{align}
	Using H\"older's and Young's inequalities, we estimate the term  $2|\langle(\B(\u)-\B(\v)),\w\rangle|=2|\langle\B(\w,\u),\w\rangle |$ as  
	\begin{align}\label{4.31}
	2|\langle\B(\w,\u),\w\rangle |&\leq 2\|\w\|_{\V}\|\u\w\|_{\H}\leq\mu\|\w\|_{\V}^2+\frac{1}{\mu }\|\u\w\|_{\H}^2.
	\end{align}
	We take the term $\|\u\w\|_{\H}^2$ from \eqref{4.31} and use H\"older's and Young's inequalities to estimate it as
	\begin{align}\label{4.32}
	\|\u\w\|_{\H}^2	&=\int_{\mathcal{O}}|\u(x)|^2|\w(x)|^2\d x\nonumber\\&=\int_{\mathcal{O}}|\u(x)|^2|\w(x)|^{\frac{4}{r-1}}|\w(x)|^{\frac{2(r-3)}{r-1}}\d x\nonumber\\&\leq\left(\int_{\mathcal{O}}|\u(x)|^{r-1}|\w(x)|^2\d x\right)^{\frac{2}{r-1}}\left(\int_{\mathcal{O}}|\w(x)|^2\d x\right)^{\frac{r-3}{r-1}}\nonumber\\&\leq{\beta\mu }\left(\int_{\mathcal{O}}|\u(x)|^{r-1}|\w(x)|^2\d x\right)+\frac{r-3}{r-1}\left(\frac{2}{\beta\mu (r-1)}\right)^{\frac{2}{r-3}}\left(\int_{\mathcal{O}}|\w(x)|^2\d x\right),
	\end{align}
	for $r>3$. Combining \eqref{4.30}, \eqref{4.31}  and \eqref{4.32}, we obtain 
	\begin{align}\label{4.33}
	&-2	\beta	\langle\mathcal{C}(\u)-\mathcal{C}(\v),\w\rangle-2\langle\B(\w,\u),\w\rangle\leq \frac{r-3}{\mu(r-1)}\left(\frac{2}{\beta\mu (r-1)}\right)^{\frac{2}{r-3}}\|\w\|_{\H}^2.
	\end{align}
	Thus, a calculation similar to \eqref{4.25} yields 
	\begin{align}\label{4.34}
	&\frac{\d}{\d t}	\|\w(t)\|_{\H}^2+\mu\lambda_{N_0}\|\w(t)\|_{\H}^2\leq\frac{r-3}{\mu(r-1)}\left(\frac{2}{\beta\mu (r-1)}\right)^{\frac{2}{r-3}}\|\w(t)\|_{\H}^2,
	\end{align}
	for a.e. $t\in[0,T]$.	For $\mu^2\lambda_{N_0}>\frac{r-3}{r-1}\left(\frac{2}{\beta\mu (r-1)}\right)^{\frac{2}{r-3}}$, we obtain 
	\begin{align}\label{4.35}
	\|\w(t)\|_{\H}^2\leq \|\w(0)\|_{\H}^2\exp\left\{-\left(\mu\lambda_{N_0}-\frac{r-3}{\mu(r-1)}\left(\frac{2}{\beta\mu (r-1)}\right)^{\frac{2}{r-3}}\right)t\right\}, 
	\end{align}
	for all $t\in[0,T]$. Taking expectation with respect to $\widetilde{\mathbb{P}}$, we get 
	\begin{align}\label{435}
	\E_{\widetilde{\mathbb{P}}}\left[\|\w(t)\|_{\H}^2\right]\leq \E_{\widetilde{\mathbb{P}}}\left[\|\w(0)\|_{\H}^2\right]\exp\left\{-\left(\mu\lambda_{N_0}-\frac{r-3}{\mu(r-1)}\left(\frac{2}{\beta\mu (r-1)}\right)^{\frac{2}{r-3}}\right)t\right\}, 
	\end{align}
	and the estimate \eqref{4p18} follows. Note  that $\{\lambda_k\}_{k=1}^{\infty}$ is  an unbounded increasing sequence,  and thus we can choose $N_0$ large enough so that $\lambda_{N_0}>\frac{(r-3)}{\mu^2(r-1)}\left(\frac{2}{\beta\mu (r-1)}\right)^{\frac{2}{r-3}}$, which completes the proof for $n=2,3$ and $r\in(3,\infty)$.

	For $n=r=3$, 	from \eqref{2.23}, we have 
	\begin{align}\label{4.36}
	-2	\beta	\langle\mathcal{C}(\u)-\mathcal{C}(\v),\w\rangle \leq- \beta\|\u\w\|_{\H}^2, 
	\end{align}
	and a calculation similar to \eqref{4.31} gives 
	\begin{align}\label{4.37}
	2|\langle\B(\w,\u),\w\rangle |&\leq 2\|\w\|_{\V}\|\u\w\|_{\H}\leq \frac{1}{\beta}\|\w\|_{\V}^2+\beta\|\u\w\|_{\H}^2.
	\end{align}
	Combining \eqref{4.36} and \eqref{4.37}, we obtain 
	\begin{align}\label{4.38}
	&-2\mu\langle\A\w,\w\rangle-2	\beta	\langle\mathcal{C}(\u)-\mathcal{C}(\v),\w\rangle-2\langle\B(\w,\u),\w\rangle\leq \left(2\mu-\frac{1}{\beta}\right)\|\w\|_{\V}^2.
	\end{align}
	Similar \eqref{4.24}, we estimate $\left(2\mu-\frac{1}{\beta}\right)\|\w\|_{\V}^2+\mu\lambda_{N_0}\|\mathrm{P}_{N_0}\w\|_{\H}^2$ as
	\begin{align}\label{424} 
	&\left(2\mu-\frac{1}{\beta}\right)\|\w\|_{\V}^2+\mu\lambda_{N_0}\|\mathrm{P}_{N_0}\w\|_{\H}^2\nonumber\\&\geq \left(2\mu-\frac{1}{\beta}\right)\|(\mathrm{I}-\mathrm{P}_{N_0})\w\|_{\V}^2+\mu\lambda_{N_0}\|\mathrm{P}_{N_0}\w\|_{\H}^2\nonumber\\&\geq \left(\mu-\frac{1}{\beta}\right)\lambda_{N_0}\|(\mathrm{I}-\mathrm{P}_{N_0})\w\|_{\H}^2+\mu\lambda_{N_0}\left(\|\mathrm{P}_{N_0}\w\|_{\H}^2+\|(\mathrm{I}-\mathrm{P}_{N_0})\w\|_{\H}^2\right)\nonumber\\&\geq \mu\lambda_{N_0}\|\w\|_{\H}^2,
	\end{align}
	for $\beta\mu>1$. 	Thus, a calculation similar to \eqref{4.25}  gives 
	\begin{align}\label{4.39}
	&\frac{\d}{\d t}	\|\w(t)\|_{\H}^2+\mu\lambda_{N_0}\|\w(t)\|_{\H}^2\leq 0, 
	\end{align}
	for a.e. $t\in[0,T]$ and hence we get 
	\begin{align}
	\|\w(t)\|_{\H}^2\leq\|\w(0)\|_{\H}^2e^{-\mu\lambda_{N_0}t}, 
	\end{align}
	for all $t\in[0,T]$ and the estimate \eqref{4a18} follows.
\end{proof}

We are bow ready to prove Theorem \ref{maint1} and Corollary \ref{cor4p4}. 
\begin{proof}[Proof of Theorem \ref{maint1}]
For $n=2,3$ and $r\in(3,\infty)$, one can get a proof  in a similar way as that of the proof of Theorem \ref{maint} except for the estimate 
\begin{align}
\E\left[\Phi(t)\log\Phi(t)\right]\leq \frac{C_{\sigma}^2\mu^2\lambda_{N_0}^2}{8\left(\mu\lambda_{N_0}-\frac{r-3}{\mu(r-1)}\left(\frac{2}{\beta\mu (r-1)}\right)^{\frac{2}{r-3}}\right)} \|\x-\y\|_{\H}^2,
\end{align}
for $\mu^2\lambda_{N_0}>\frac{r-3}{r-1}\left(\frac{2}{\beta\mu (r-1)}\right)^{\frac{2}{r-3}}$ and hence the estimate \eqref{4p2} follows. 

For $n=r=3$ and $\beta\mu> 1$, the above estimate has to be replaced by 
\begin{align}
\E\left[\Phi(t)\log\Phi(t)\right]\leq \frac{C_{\sigma}^2\mu\lambda_{N_0}}{8} \|\x-\y\|_{\H}^2,
\end{align}
and one can easily get the estimate \eqref{4p2}.
\end{proof}

\begin{proof}[Proof of Corollary \ref{cor4p4}]
	For the cases $n=2,3$ and $r\in(3,\infty)$, we define 
	\begin{align*}\Theta(\x,\y)=\wi\gamma \|\x-\y\|_{\H}^2\ \text{  and }\ \Psi_t(\x,\y)=e^{-\wi{\theta}t}\|\x-\y\|_{\H}. \end{align*}  In order to prove the  existence of an invariant measure for  the transition semigroup $\{\mathrm{P}_t\}_{t\geq 0},$ we can follow similar arguments as in the proof of Corollary \ref{cor4.4}. Here we sketch important steps only. A calculation similar to \eqref{4.34} yields 
	\begin{align}\label{4.51}
	&\frac{\d}{\d t}	\|\u(t,\x)-\u(t,\y)\|_{\H}^2+\mu\|\u(t,\x)-\u(t,\y)\|_{\V}^2\nonumber\\&\leq\frac{r-3}{\mu(r-1)}\left(\frac{2}{\beta\mu (r-1)}\right)^{\frac{2}{r-3}}\|\u(t,\x)-\u(t,\y)\|_{\H}^2,
	\end{align}
	for a.e. $t\in[0,T]$. An application Gronwall's inequality yields 
	\begin{align}
	\|\u(t,\x)-\u(t,\y)\|_{\H}^2\leq \|\x-\y\|_{\H}^2\exp\left\{\left(\frac{r-3}{\mu(r-1)}\left(\frac{2}{\beta\mu (r-1)}\right)^{\frac{2}{r-3}}\right)t\right\},
	\end{align}
	and proceeding similarly in the proof of Corollary \ref{cor4.4}, we get the required result. 
	
	For the case $n=3$, $r=3$ and $\beta\mu>1$, we obtain 
	\begin{align}
	\|\u(t,\x)-\u(t,\y)\|_{\H}^2\leq e^{-\mu\lambda_1 t}\|\x-\y\|_{\H}^2,
	\end{align}
	and the existence of an invariant measure for the transition semigroup $\{\mathrm{P}_t\}_{t\geq 0}$ can be proved in a similar way as in the case of $n=2$ and $r\in[1,3]$. 
\end{proof}

\section{The SCBF equations perturbed by multiplicative noise} \label{sec5}\setcounter{equation}{0} 
In this section, we consider the following stochastic  convective Brinkman-Forchheimer  equations perturbed by degenerate multiplicative  Gaussian noise:
\begin{equation}\label{51}
\left\{
\begin{aligned}
\d\u(t)+[\mu \A\u(t)+\B(\u(t))+\beta\mathcal{C}(\u(t))]\d t&=\sigma(\u)\d\W(t), \ t\in(0,T),\\
\u(0)&=\x,
\end{aligned}
\right.
\end{equation}
where $\x\in\H$. Due to technical difficulties, we consider the cases $n=2,3$ and $r\in[3,\infty)$ only. In the degenerate multiplicative noise case, the  asymptotic  log-Harnack  inequality  for  several  kinds  of  models  on  stochastic differential systems with infinite memory   is established in \cite{JBFY} and  for a class of semilinear SPDEs  is obtained in \cite{WHSL1}.  

Let $\mathcal{L}_2(\H)=\mathcal{L}_2(\H,\H)$ be the space of all Hilbert-Schmidt operators from $\H$ to $\H$. We need the following assumptions on the noise coefficient to obtain our main results. 
\begin{assumption}\label{ass4.9}
	\begin{enumerate}
		\item [(A1)]
	The mapping $\sigma:\V\to\mathcal{L}_2(\H)$ is bounded and Lipschitz, that is, 
	\begin{align}
	\sup_{\u\in\V}\|\sigma(\u)\|_{\mathcal{L}_2(\H)}<\infty,
\end{align}
and  for all $\u_1,\u_2\in\H$, there exists a constant $L>0$ such that 
\begin{align}
\|\sigma(\u_1)-\sigma(\u_2)\|_{\mathcal{L}_2(\H)}^2\leq L\|\u_1-\u_2\|_{\H}^2.
\end{align}
\item [(A2)] The exists a constant $N_0\in\mathbb{N}$ so that for all $\u\in\H$, we have $\mathrm{P}_{N_0}\H\subset\mathrm{Range}(\sigma(\u))$  and $\sigma(\u)\v=0$ if $\v\in(\mathrm{I}-\mathrm{P}_{N_0})\H$. Furthermore, the corresponding pseudo-inverse operator $\sigma(\u)^{-1}:\mathrm{P}_{N_0}\H\to  \mathrm{P}_{N_0}\H$ is uniformly bounded, that is, there exists a constant $\widetilde{K}>0$ such that 
\begin{align}
\sup_{\u\in\H}\|\sigma(\u)^{-1}\|_{\mathcal{L}(\mathrm{P}_{N_0}(\H))}\leq \wi K, 
\end{align}
where $\mathcal{L}(\mathrm{P}_{N_0}(\H))=\mathcal{L}(\mathrm{P}_{N_0}(\H),\mathrm{P}_{N_0}(\H))$ is the space of all bounded linear operators from $\mathrm{P}_{N_0}(\H)$ to $\mathrm{P}_{N_0}(\H)$. 
	\end{enumerate}
\end{assumption}
Under the above assumptions, we have the following existence and uniqueness of strong solution to the system \eqref{51}. 
\begin{theorem}[Theorem 3.7, \cite{MTM8}]\label{exis1}
	Let $\x\in\H$, for $n=2,3$ and $r\geq 3$ be given ($2\beta\mu\geq 1,$ for $n=r=3$).  Then there exists a \emph{pathwise unique strong solution}
	$\u(\cdot)$ to the system (\ref{32}) such that \begin{align*}\u&\in\mathrm{L}^2(\Omega;\mathrm{L}^{\infty}(0,T;\H)\cap\mathrm{L}^2(0,T;\V))\cap\mathrm{L}^{r+1}(\Omega;\mathrm{L}^{r+1}(0,T;\widetilde{\L}^{r+1})),\end{align*} with $\mathbb{P}$-a.s., continuous modification with trajectories in $\H$ (that is, $\u\in\C([0,T];\H)\cap\mathrm{L}^2(0,T;\V))\cap\mathrm{L}^{r+1}(\Omega;\mathrm{L}^{r+1}(0,T;\widetilde{\L}^{r+1})$, $\mathbb{P}$-a.s.) satisfying 
	\begin{align}
	\E\left[\sup_{t\in[0,T]}\|\u(t)\|_{\H}^2+2\mu \int_0^T\|\u(t)\|_{\V}^2\d t+2\beta\int_0^T\|\u(t)\|_{\widetilde{\L}^{r+1}}^{r+1}\d t\right]\leq C(L,T)\|\x\|_{\H}^2.
	\end{align}
	Moreover, the strong solution $\u(\cdot)$ satisfies the following It\^o formula: 
	\begin{align}
	&	\|\u(t)\|_{\H}^2+2\mu \int_0^t\|\u(s)\|_{\V}^2\d s+2\beta\int_0^t\|\u(s)\|_{\widetilde{\L}^{r+1}}^{r+1}\d s\nonumber\\&=\|{\u_0}\|_{\H}^2+\int_0^t\Tr(\sigma(\u)\sigma^*(\u))\d s+2\int_0^t(\sigma(\u)\d\W(s),\u(s)),
	\end{align}
	for all $t\in[0,T]$, $\mathbb{P}$-a.s.
\end{theorem}

The main results of this section are the following: 

\begin{theorem}\label{maint2}
	Let $\u(t,\x)$ be the unique strong solution to the system \eqref{32}	with the initial data $\x\in\H$. Suppose that the noise co-efficient $\sigma(\cdot)$ satisfies Assumption \ref{ass4.9}. Let $n=2,3$, $r\in(3,\infty)$ and $\lambda_{N_0}>\left(\frac{r-3}{\mu^2(r-1)}\left(\frac{2}{\beta\mu (r-1)}\right)^{\frac{2}{r-3}}+\frac{L}{\mu}\right)$.  Then, for any $\x,\y\in\H$ and for any $f\in\mathscr{B}^+_b(\H)$ with $\|\nabla\log f\|_{\infty}<\infty$, we have the following \emph{asymptotic log-Harnack inequality}:
	\begin{align}\label{4b2}
	\mathrm{P}_t\log f(\y)&\leq \log\mathrm{P}_tf(\x)+\widehat{\gamma}\|\x-\y\|_{\H}^2+e^{-\widehat{\theta}t}\|\x-\y\|_{\H}\|\nabla\log f\|_{\infty}, \ t >0, 
	\end{align}
	where the constants  \ \begin{align*}\widehat{\theta}&=\frac{1}{2}\left[\mu\lambda_{N_0}-\left(\frac{r-3}{\mu(r-1)}\left(\frac{2}{\beta\mu (r-1)}\right)^{\frac{2}{r-3}}+L\right)\right]\ \text{ and }\\  \widehat{\gamma}&=\frac{\wi{K}^2\mu^2\lambda_{N_0}^2}{8\left[\mu\lambda_{N_0}-\left(\frac{r-3}{\mu(r-1)}\left(\frac{2}{\beta\mu (r-1)}\right)^{\frac{2}{r-3}}+L\right)\right]} .
	\end{align*}
	
	For $n=2,3$, $r=3$, $\beta\mu>1$ and $\lambda_{N_0}>\frac{L}{\mu}$, we have the following \emph{asymptotic log-Harnack inequality}:
	\begin{align}\label{4c2}
	\mathrm{P}_t\log f(\y)&\leq \log\mathrm{P}_tf(\x)+ \frac{\widetilde{K}\mu^2\lambda_{N_0}^2}{8(\mu\lambda_{N_0}-L)}\|\x-\y\|_{\H}^2+e^{-\frac{(\mu\lambda_{N_0}-L)}{2}t}\|\x-\y\|_{\H}\|\nabla\log f\|_{\infty}, 
	\end{align}
	for $t>0$. 
\end{theorem}
Analogous to Corollaries \ref{cor4.4} and \ref{cor4p4}, we have the following result for $n=2,3$ and $r\in(3,\infty)$. For the case $n=2,3$, $r=3$ and $\beta\mu>1$, one has to replace $\widehat{\theta}$ with $\frac{\mu\lambda_{N_0}-L}{2}$ and $\widehat{\gamma}$ with $ \frac{\widetilde{K}\mu^2\lambda_{N_0}^2}{8(\mu\lambda_{N_0}-L)}$ in the Corollary given below. 
\begin{corollary}\label{cor4a4} 
	Under the assumptions of Theorem \ref{maint}, we have 
	\begin{enumerate}
		\item [(1)] \textbf{Gradient estimate.} For any $\y\in\H$, $t>0$ and $f\in\mathrm{Lip}_b(\H):=\mathrm{Lip}(\H)\cap\mathscr{B}_b(\H),$ we have 
		\begin{align}
		\|\nabla\mathrm{P}_tf(\y)\|_{\H}\leq \sqrt{\widehat{\gamma}}\sqrt{\mathrm{P}_tf^2(\y)-(\mathrm{P}_tf)^2(\y)}+e^{-\widehat{\theta}t}\|\nabla f\|_{\infty}.
		\end{align}
		In particular, $\mathrm{P}_t$ is asymptotically strong Feller. 
		\item [(2)] \textbf{Asymptotic irreducibility.} Let $\y\in\H$ and $\A\subset\H$ be a measurable set such that $$\delta(\y,\A):= \liminf_{t\to\infty}\mathrm{P}_t(\y,\A)>0.$$ Then, we have $$\liminf_{t\to\infty}\mathrm{P}_t(\x,\A_{\e})>0, \ \text{ for any } \ \x\in\H, \ \e>0,$$ where $\A_{\e}:=\{\x\in\H:\rho(\x,\A)<\e\}$ with $\rho(\x,\A):=\inf\limits_{\z\in\A}\|\x-\z\|_{\H}.$ Furthermore, for any $\e_0\in(0,\delta(\y,\A))$, there exists a constant $t_0>0$ such that $$\mathrm{P}_t(\x,\A_{\e})>0, \  \text{ provided } \  t\geq t_0\ \text{ and }\ e^{-\widehat{\theta}t}\|\x-\y\|_{\H}<\e\e_0.$$ 
		\item [(3)] \textbf{Asymptotic heat kernel estimate.} If  $\mathrm{P}_t$ has an invariant probability measure $\nu$,  then for any $f\in\mathscr{B}_b^+(\H)$ with $\|\nabla f\|_{\infty}<\infty$, we have 
		$$\limsup_{t\to\infty}\mathrm{P}_tf(\y)\leq\log\left(\frac{\nu(e^{f})}{\int_{\H}e^{-\widehat{\gamma}\|\x-\y\|_{\H}^2}\nu(\d \x)}\right), \ \y\in\H. $$ Consequently, for any closed set $\A\subset\H$ with $\nu(\A)=0$, we have $$\lim_{t\to\infty}\mathrm{P}_t\chi_{\A}(\y)=0,\ \text{ for all }\ \y\in\H.$$
		\item [(4)] \textbf{Uniqueness of invariant probability measure.} The Markovian transition semigroup $\mathrm{P}_t$ has at most one invariant probability measure. 
	\end{enumerate}
\end{corollary}

As in the case of  additive Gaussian noise, we use the coupling method to prove Theorem \ref{maint2}. Given a $\y\in\H$, we consider the following systems:
\begin{equation}\label{459}
\left\{
\begin{aligned}
\d\u(t)&=-[\mu \A\u(t)+\B(\u(t))+\beta\mathcal{C}(\u(t))]\d t+\sigma(\u(t))\d\W(t),\ \u(0)=\x\in\H, \\
\d\v(t)&=-[\mu\A\v(t)+\B(\v(t))+\beta\mathcal{C}(\v(t))]\d t+\sigma(\v(t))\d\W(t)\\&\quad+\frac{\mu\lambda_{N_0}}{2}\sigma(\v(t))\sigma(\u(t))^{-1}(\u(t)-\v(t))^l\d t,\ \v(0)=\y\in\H. 
\end{aligned}
\right. 
\end{equation} 
Under the Assumption \ref{ass4.9} (A1) and (A2), one can easily see that the additional drift term $\sigma(\v(t))\sigma(\u(t))^{-1}(\u(t)-\v(t))^l$ satisfies the local monotonicity as well as hemicontinuity properties. Hence the coupling $(\u(\cdot),\v(\cdot))$ is well-defined. In order to investigate that $\v(t)$ has a transition semigroup $\mathrm{P}_t$ under the weighted probability measure $\widetilde{\mathbb{P}}$, we let 
\begin{align}
\widetilde{\W}(t):=\W(t)+\int_0^t\wi h(s)\d s, \ \wi h(t):=\frac{\mu\lambda_{N_0}}{2}\sigma(\u(t))^{-1}(\u(t)-\v(t))^l, \ t\geq 0. 
\end{align}
We also define 
\begin{align}
\wi\Phi(t):=\exp\left\{-\int_0^t(\wi h(s),\d\W(s))-\frac{1}{2}\int_0^t\|\wi h(s)\|_{\H}^2\d s\right\}, \ t\geq 0. 
\end{align}
Then, we have the following result. 

\begin{lemma}
	For $n=2,3$ and $r\in(3,\infty)$, $\lambda_{N_0}>\frac{r-3}{\mu^2(r-1)}\left(\frac{2}{\beta\mu (r-1)}\right)^{\frac{2}{r-3}}+\frac{L}{\mu},$  and any $T>0$, we have 
	\begin{align}\label{4.62}
	\E\left[\wi\Phi(t)\log\wi\Phi(t)\right]\leq \frac{\mu^2\lambda_{N_0}^2\widetilde{K}^2}{8[\mu\lambda_{N_0}-(\upeta+L)]} \|\x-\y\|_{\H}^2, \ \text{ for all }\ 0\leq t\leq T,
	\end{align}
	where $\upeta=\frac{r-3}{\mu(r-1)}\left(\frac{2}{\beta\mu (r-1)}\right)^{\frac{2}{r-3}}$. For $r=3$, $\mu\beta>1$ and $\lambda_{N_0}>\frac{L}{\mu}$, we have 
	\begin{align}\label{462}
	\E\left[\wi\Phi(t)\log\wi\Phi(t)\right]\leq \frac{\mu^2\lambda_{N_0}^2\widetilde{K}^2}{8[\mu\lambda_{N_0}-L]} \|\x-\y\|_{\H}^2, \ \text{ for all }\  0\leq t\leq T.
	\end{align}
\end{lemma}
\begin{proof}
	For any $N\geq \|\x\|_{\H}+\|\y\|_{\H}$, we define a sequence of stopping times by 
	\begin{align}
	\tau_N:=\inf_{t\geq 0}\left\{t: \|\u(t)\|_{\H}+\|\v(t)\|_{\H}>N\right\}. 
	\end{align}
	From the definition of $\tau_N$, it is clear that $\tau_N\uparrow\infty$ as $N\uparrow\infty$, and $\{\|\u(t)\|_{\H}\}_{t\in[0,T\wedge\tau_N]}$ and $\{\|\v(t)\|_{\H}\}_{t\in[0,T\wedge\tau_N]}$ are bounded. The Assumption \ref{ass4.9} (A2) implies 
	\begin{align*}
	\E\left[\int_0^{\T}\|\wi h(t)\|_{\H}^2\d t\right]&\leq \frac{\mu^2\lambda_{N_0}^2}{4}\E\left[\int_0^{\T}\|\sigma(\u(t))^{-1}\|_{\mathcal{L}(\mathrm{P}_{N_0}(\H))}^2\|(\u(t)-\v(t))^l\|_{\mathrm{P}_{N_0}(\H)}^2\d t\right]\nonumber\\&\leq\frac{\mu^2\lambda_{N_0}^2\wi K^2}{2}\left\{\E\left[\int_0^{\T}\|\u(t)\|_{\H}^2\d t\right]+\E\left[\int_0^{\T}\|\v(t)\|_{\H}^2\d t\right]\right\}\nonumber\\&\leq {\mu^2\lambda_{N_0}^2\wi K^2N^2T}<\infty, 
	\end{align*}
	 which implies Novikov's condition on $[0,T\wedge\tau_N]$. Then, applying  Girsanov’s theorem, we infer that $\{\wi\W(t)\}_{t\in[0,\T]}$  is a cylindrical Wiener process under the weighted probability measure $\widetilde{\mathbb{P}}_{T,N}=\wi\Phi(\T)\mathbb{P}$. Thus, we rewrite the coupling $(\u(\cdot),\v(\cdot))$ given in \eqref{459} as 
	\begin{equation}\label{464}
	\left\{
	\begin{aligned}
	\d\u(t)&=-\left[\mu \A\u(t)+\B(\u(t))+\beta\mathcal{C}(\u(t))+\frac{\mu\lambda_{N_0}}{2}(\u(t)-\v(t))^l\right]\d t\\&\quad+\sigma(\u(t))\d\wi\W(t), \ t\in[0,\T], \\
	\d\v(t)&=-[\mu\A\v(t)+\B(\v(t))+\beta\mathcal{C}(\v(t))]\d t+\sigma(\v(t))\d\wi\W(t), \ t\in[0,\T], 
	\end{aligned}
	\right. 
	\end{equation} 
	with the corresponding initial values $\u(0)=\x$ and $\v(0)=\y$. Using the infinite dimensional It\^o formula to the process $\|\w(\cdot)\|_{\H}^2=\|\u(\cdot)-\v(\cdot)\|_{\H}^2$ under the probability measure $\widetilde{\mathbb{P}}_{T,N}$ and for $t\in[0,\T]$, we obtain 
	\begin{align}\label{4.65}
&	\|\w(t)\|_{\H}^2+2\mu\int_0^t\|\w(s)\|_{\V}^2\d s\nonumber\\&=\|\w(0)\|_{\H}^2-2\int_0^t\langle\B(\u(s))-\B(\v(s)),\w(s)\rangle\d s-2\beta\int_0^t\langle\mathcal{C}(\u(s))-\mathcal{C}(\v(s)),\w(s)\rangle\d s\nonumber\\&\quad+\mu\lambda_{N_0}\int_0^t((\w(s))^l,\w(s))\d s+\int_0^t\|\sigma(\u(s))-\sigma(\v(s))\|_{\mathcal{L}_2(\H)}^2\d s\nonumber\\&\quad+2\int_0^t((\sigma(\u(s))-\sigma(\v(s)))\d\wi\W(s),\w(s)),\ \widetilde{\mathbb{P}}_{T,N}\text{-a.s.},
	\end{align}
	for all $t\in[0,\T]$. For $n=2,3$ and $r(3,\infty)$, using \eqref{4.24} and \eqref{4.33} in \eqref{4.65}, we find 
	\begin{align}\label{4.66} 
	&\|\w(t)\|_{\H}^2+\mu\lambda_{N_0}\int_0^t\|\w(s)\|_{\H}^2\d s\nonumber\\&\leq\|\x-\y\|_{\H}^2+\frac{r-3}{\mu(r-1)}\left(\frac{2}{\beta\mu (r-1)}\right)^{\frac{2}{r-3}}\int_0^t\|\w(s)\|_{\H}^2\d s+L\int_0^t\|\w(s)\|_{\H}^2\d s\nonumber\\&\quad+2\int_0^t((\sigma(\u(s))-\sigma(\v(s)))\d\wi\W(s),\w(s)), 
	\end{align}
		for all $t\in[0,\T]$. Taking expectation in \eqref{4.66} and using the fact that final term appearing in the right hand side of the inequality \eqref{4.66} is a local martingale with zero expectation, we deduce that
		\begin{align}
	&\E_{\widetilde{\mathbb{P}}_{T,N}}\left[\|\w(t)\|_{\H}^2\right]\nonumber\\&\leq\|\x-\y\|_{\H}^2+\left[\frac{r-3}{\mu(r-1)}\left(\frac{2}{\beta\mu (r-1)}\right)^{\frac{2}{r-3}}+L-\mu\lambda_{N_0}\right]\int_0^t\E_{\widetilde{\mathbb{P}}_{T,N}}\left[\|\w(s)\|_{\H}^2\right]\d s,	
		\end{align}
		where we used \eqref{4.24} and Fubini's theorem. Applying Gronwall's inequality, we get 
		\begin{align}\label{4.68}
		\E_{\widetilde{\mathbb{P}}_{T,N}}\left[\|\w(t)\|_{\H}^2\right]\leq\|\x-\y\|_{\H}^2e^{-[\mu\lambda_{N_0}-(\upeta+L)]t}, \ \text{ for all }\ t\in[0,\T],
		\end{align}
		where $\upeta=\frac{r-3}{\mu(r-1)}\left(\frac{2}{\beta\mu (r-1)}\right)^{\frac{2}{r-3}}$. Using the definition of 
		$\wi\Phi(\cdot)$, $\wi\W(\cdot)$ and the Assumption \ref{ass4.9} (A2), we have 
		\begin{align}\label{4.69}
		\E\left[\wi\Phi(\t)\log\wi\Phi(\t)\right]&=\E_{\widetilde{\mathbb{P}}_{T,N}}\left[\log\wi\Phi(\t)\right]\nonumber\\&=\frac{1}{2}\E_{\widetilde{\mathbb{P}}_{T,N}}\left[\int_0^{\t}\|\wi h(s)\|_{\H}^2\d s\right]\nonumber\\&\leq\frac{\mu^2\lambda_{N_0}^2\widetilde{K}^2}{8}\int_0^{t}\E_{\widetilde{\mathbb{P}}_{T,N}}\left[\|\u(s)-\v(s)\|_{\H}^2\right]\d s,
		\end{align}
		where we used  Fubini's theorem in the final step. For $\mu\lambda_{N_0}>\upeta+L$, using \eqref{4.68} in \eqref{4.69}, we deduce that 
		\begin{align}\label{470}
			\E\left[\wi\Phi(\t)\log\wi\Phi(\t)\right]&\leq \frac{\mu^2\lambda_{N_0}^2\widetilde{K}^2}{8[\mu\lambda_{N_0}-(\upeta+L)]}\|\x-\y\|_{\H}^2, \ 0\leq t\leq T. 
			\end{align}
		For $0\leq s<t\leq T$, using the dominated convergence theorem and the martingale property of $\{\wi\Phi(\t)\}_{t\in[0,T]}$, we get 
		\begin{align}
		\E\left(\wi\Phi(t)\big|\mathscr{F}_s\right)=\E\left(\lim_{n\to\infty}\wi\Phi(\t)\big|\mathscr{F}_s\right)=\lim_{n\to\infty}\E\left(\wi\Phi(\t)\big|\mathscr{F}_s\right)=\wi\Phi(s), 
		\end{align} 
		so that $\{\wi\Phi(t)\}_{t\in[0,T]}$ is a martingale. Let us define $\widetilde{\mathbb{P}}_T:=\wi\Phi(T)\mathbb{P}$. Since $\{\wi\Phi(t)\}_{t\in[0,T]}$ is a martingale, we obtain $\widetilde{\mathbb{P}}_T(\A)=\widetilde{\mathbb{P}}_{T,N}(\A)$, if $\A\in\mathscr{F}_{T\wedge\tau_N}$. Note also that $\{\wi\W(t)\}_{t\in[0,T]}$ is a cylindrical Wiener process under the probability measure $\widetilde{\mathbb{P}}_T$. Applying Fatou’s lemma, we obtain 
		\begin{align}\label{472}
		\liminf_{N\to\infty}\E_{\widetilde{\mathbb{P}}_{T,N}}\left[\log\wi\Phi(\t)\right]&=	\liminf_{N\to\infty}\E_{\widetilde{\mathbb{P}}_{T}}\left[\log\wi\Phi(\t)\right]\nonumber\\&=\liminf_{N\to\infty}\frac{1}{2}\E_{\widetilde{\mathbb{P}}_{T}}\left[\int_0^{\t}\|\wi h(s)\|_{\H}^2\d s\right]\nonumber\\&\geq\frac{1}{2}\E_{\widetilde{\mathbb{P}}_{T}}\left[\liminf_{N\to\infty}\int_0^{\t}\|\wi h(s)\|_{\H}^2\d s\right]=\frac{1}{2}\E_{\widetilde{\mathbb{P}}_{T}}\left[\int_0^{t}\|\wi h(s)\|_{\H}^2\d s\right],
		\end{align}
		for any $t\in[0,T]$. Combining \eqref{470} and \eqref{472}, we further have 
		\begin{align}
		\E\left[\wi\Phi(t)\log\wi\Phi(t)\right]&=\E_{\widetilde{\mathbb{P}}_{T}}\left[\log\wi\Phi(t)\right]=\frac{1}{2}\E_{\widetilde{\mathbb{P}}_{T}}\left[\int_0^{t}\|\wi h(s)\|_{\H}^2\d s\right]\nonumber\\&\leq \liminf_{N\to\infty}\E_{\widetilde{\mathbb{P}}_{T,N}}\left[\log\wi\Phi(\t)\right]\nonumber\\&\leq \frac{\mu^2\lambda_{N_0}^2\widetilde{K}^2}{8[\mu\lambda_{N_0}-(\upeta+L)]}\|\x-\y\|_{\H}^2, \ 0\leq t\leq T, 
		\end{align}
		which completes the proof of \eqref{4.62}. 
		
		For $r=3$ and $\beta\mu>1$, calculation similar to \eqref{424} and \eqref{4.38} gives 
			\begin{align}\label{4.75}
		&\E_{\widetilde{\mathbb{P}}_{T,N}}\left[	\|\w(t)\|_{\H}^2\right]+(\mu\lambda_{N_0}-L)\int_0^t\E_{\widetilde{\mathbb{P}}_{T,N}}\left[\|\w(s)\|_{\H}^2\right]\d s\leq\|\x-\y\|_{\H}^2,
		\end{align}
		for all $t\in[0,\T]$ and the estimate \eqref{462} follows for $\mu>\frac{L}{\lambda_{N_0}}$, by proceeding similarly as in the previous case. 
\end{proof}

\begin{proof}[Proof of Theorem \ref{maint2}]
	For $n=2,3$ and $r\in(3,\infty)$, one can obtain a proof  in a similar way as that of the proof of Theorems \ref{maint} and \ref{maint1}. From \eqref{4.62}, we have the estimate  
	\begin{align}\label{5.27}
	\E\left[\wi\Phi(t)\log\wi\Phi(t)\right]\leq \frac{\wi{K}^2\mu^2\lambda_{N_0}^2}{8\left[\mu\lambda_{N_0}-\left(\frac{r-3}{\mu(r-1)}\left(\frac{2}{\beta\mu (r-1)}\right)^{\frac{2}{r-3}}+L\right)\right]} \|\x-\y\|_{\H}^2,
	\end{align}
	for $\mu\lambda_{N_0}>\frac{r-3}{\mu(r-1)}\left(\frac{2}{\beta\mu (r-1)}\right)^{\frac{2}{r-3}}+L$. Using  similar arguments to \eqref{4.68}, we can get the following estimate:
	\begin{align}\label{5.28}
	\E_{\widetilde{\mathbb{P}}_t}\left[\|\u(t,\x)-\v(t,\x)\|_{\H}^2\right]\leq \|\x-\y\|_{\H}^2e^{-\left[\mu\lambda_{N_0}-\left(\frac{r-3}{\mu(r-1)}\left(\frac{2}{\beta\mu (r-1)}\right)^{\frac{2}{r-3}}+L\right)\right]t}, \ \text{ for all }\ t\in[0,T]. 
	\end{align}
	Hence, for any $f\in\mathscr{B}_b^+(\H)$ with $\|\nabla\log f\|_{\infty}<\infty$, using \eqref{4.42}, \eqref{5.27} and \eqref{5.28}, we have 
	\begin{align}
	\mathrm{P}_t[\log f(\y)]&=  \E_{\widetilde{\mathbb{P}}_t}\left[\log f(\u(t,\x))\right]+\E_{\widetilde{\mathbb{P}}_t}\left[\log f(\v(t,\y))-\log f(\u(t,\x))\right]\nonumber\\&\leq \E\left[\wi{\Phi}(t)\log f(\u(t,\x))\right]+\|\nabla\log f\|_{\infty}\E_{\widetilde{\mathbb{P}}_t}\left[\|\u(t,\x)-\v(t,\x)\|_{\H}\right]\nonumber\\&\leq 	\E\left[\wi\Phi(t)\log\wi\Phi(t)\right]+\log\mathrm{P}_tf(\x)+\|\nabla\log f\|_{\infty}\left(\E_{\widetilde{\mathbb{P}}_t}\left[\|\u(t,\x)-\v(t,\x)\|_{\H}^2\right]\right)^{1/2} \nonumber\\&\leq \frac{\wi{K}^2\mu^2\lambda_{N_0}^2}{8\left[\mu\lambda_{N_0}-\left(\frac{r-3}{\mu(r-1)}\left(\frac{2}{\beta\mu (r-1)}\right)^{\frac{2}{r-3}}+L\right)\right]} \|\x-\y\|_{\H}^2+\log\mathrm{P}_tf(\x)\nonumber\\&\quad+ e^{-\frac{1}{2}\left[\mu\lambda_{N_0}-\left(\frac{r-3}{\mu(r-1)}\left(\frac{2}{\beta\mu (r-1)}\right)^{\frac{2}{r-3}}+L\right)\right]t}\|\x-\y\|_{\H}, 
	\end{align}
	and we obtain estimate \eqref{4b2}. 
	
	For $n=2,3$, $r=3$ and $\beta\mu> 1$, the estimate \eqref{5.27}  needs to be replaced by 
	\begin{align}
	\E\left[\Phi(t)\log\Phi(t)\right]\leq \frac{\widetilde{K}\mu^2\lambda_{N_0}^2}{8(\mu\lambda_{N_0}-L)} \|\x-\y\|_{\H}^2,
	\end{align}
	and the estimate \eqref{4c2} follows. 
\end{proof}

\begin{proof}[Proof of Corollary \ref{cor4a4}]
	For the case $n=2,3$ and $r\in(3,\infty)$, we define 
	\begin{align*}\Theta(\x,\y)=\widehat\gamma \|\x-\y\|_{\H}^2\ \text{  and }\ \Psi_t(\x,\y)=e^{-\widehat{\theta}t}\|\x-\y\|_{\H}. \end{align*}  In order to prove the  existence of an invariant measure for  the transition semigroup $\{\mathrm{P}_t\}_{t\geq 0},$ we can follow similar arguments as in the proof of Corollary \ref{cor4p4}.  For $n=2,3$ and $r\in(3,\infty)$, using the infinite dimensional It\^o formula to the process $\|\u(\cdot,\x)-\u(\cdot,\y)\|_{\H}^2$, we obtain 
	\begin{align}
&	\|\u(t,\x)-\u(t,\y)\|_{\H}^2+2\mu\int_0^t\|\u(s,\x)-\u(s,\y)\|_{\V}^2\d s\nonumber\\&=\|\x-\y\|_{\H}^2-2\int_0^t\langle\B(\u(s,\x))-\B(\u(s,\y)),\u(s,\x)-\u(s,\y)\rangle\d s\nonumber\\&\quad-2\beta\int_0^t\langle\mathcal{C}(\u(s,\x))-\mathcal{C}(\u(s,\y)),\u(s,\x)-\u(s,\y)\rangle\d s\nonumber\\&\leq \|\x-\y\|_{\H}^2+\left[\frac{r-3}{\mu(r-1)}\left(\frac{2}{\beta\mu (r-1)}\right)^{\frac{2}{r-3}}+L\right]\int_0^t\|\u(s,\x)-\u(s,\y)\|_{\H}^2\d s\nonumber\\&\quad+2\int_0^t((\sigma(\u(s,\x))-\sigma(\u(s,\y)))\d\W(s),\u(s,\x)-\u(s,\y)), 
	\end{align}
$\mathbb{P}$-a.s.,	for all $t\in[0,T]$, where we used a calculation similar to \eqref{4.66}. Taking expectation, and then by using Fubini's theorem and Gronwall’s lemma, we get 
\begin{align}
\E\left[\|\u(t,\x)-\u(t,\y)\|_{\H}^2\right]\leq \|\x-\y\|_{\H}^2e^{\left[\frac{r-3}{\mu(r-1)}\left(\frac{2}{\beta\mu (r-1)}\right)^{\frac{2}{r-3}}+L\right]t}, 
\end{align}
for all $t\in[0,T]$, which implies that $\{\mathrm{P}_t\}_{t\geq 0}$ is a Feller semigroup.  Therefore, for the existence of an invariant measure, we only need to prove the tightness of $\{\nu_n:n\geq 1\}$. Using the infinite dimensional It\^o formula to the process $\|\u(\cdot)\|_{\H}^2$, we find 
\begin{align}\label{5.33}
&\|\u(t)\|_{\H}^2+2\mu\int_0^t\|\u(s)\|_{\V}^2\d s+2\beta\int_0^t\|\u(s)\|_{\wi\L^{r+1}}^{r+1}\d s\nonumber\\&=\|\x\|_{\H}^2+\int_0^t\|\sigma(\u(s))\|_{\mathcal{L}_2(\H)}^2\d s+2\int_0^t(\sigma(\u(s))\d\W(s),\u(s))\nonumber\\&\leq\|\x\|_{\H}^2+C\int_0^t(1+L\|\u(s)\|_{\H}^2)\d s+2\int_0^t(\sigma(\u(s))\d\W(s),\u(s)), 
\end{align}
$\mathbb{P}$-a.s.,	for all $t\in[0,T]$. Thus, from \eqref{5.33}, it is immediate that 
\begin{align}\label{5.34}
&\|\u(t)\|_{\H}^2\leq\|\x\|_{\H}^2+\int_0^t\left[C+\left(-2\mu+\frac{CL}{\lambda_1}\right)\|\u(s)\|_{\H}^2\right]\d s+2\int_0^t(\sigma(\u(s))\d\W(s),\u(s)).
\end{align}
For $L<\frac{2\mu\lambda_1}{C}$,  proceeding similarly in the proof of Corollary \ref{cor4.4}, we get the required result. 
	
	For the case $n=2,3$, $r=3$ and $\beta\mu>1$ and $L<\mu\lambda_1$, we obtain 
	\begin{align}
\E\left[	\|\u(t,\x)-\u(t,\y)\|_{\H}^2\right]\leq e^{-(\mu\lambda_1-L) t}\|\x-\y\|_{\H}^2,
	\end{align}
	and the existence of an invariant measure for the transition semigroup $\{\mathrm{P}_t\}_{t\geq 0}$ can be proved in a similar way as in the previous case. 
\end{proof}

 \medskip\noindent
{\bf Acknowledgments:} M. T. Mohan would  like to thank the Department of Science and Technology (DST), India for Innovation in Science Pursuit for Inspired Research (INSPIRE) Faculty Award (IFA17-MA110).

\end{document}